\documentclass{siamltex}
\usepackage{amsmath}
\usepackage{amsfonts}
\usepackage{amssymb}
\usepackage{graphicx}
\usepackage{color}
\usepackage{bm}

\allowdisplaybreaks

\newtheorem{remark}[theorem]{Remark}
\newtheorem{assumption}[theorem]{Assumption}

\newcommand{\be}{\begin{equation}}
\newcommand{\ee}{\end{equation}}
\newcommand{\bea}{\begin{eqnarray}}
\newcommand{\eea}{\end{eqnarray}}
\newcommand{\beas}{\begin{eqnarray*}}
\newcommand{\eeas}{\end{eqnarray*}}

\input{tcilatex}

\begin{document}

\title{An ensemble-Proper Orthogonal\\ Decomposition method for the\\  Nonstationary Navier-Stokes Equations\thanks{Supported by the US Air Force Office of Scientific Research grant FA9550-15-1-0001 and US Department of Energy Office of Science grants DE-SC0009324 and DE-SC0010678.}}

\author{
Max Gunzburger${}^\dagger$, Nan Jiang${}^\dagger$, and Michael Schneier\thanks{Department of Scientific Computing,  
Florida State University,
Tallahassee, FL 32306-4120
({\tt mgunzburger@fsu.edu}, {\tt njiang@fsu.edu}, {\tt mschneier89@gmail.com}).
}
}
\maketitle

\begin{abstract}
The definition of partial differential equation (PDE) models usually involves a set of parameters whose values may vary over a wide range. The solution of even a single set of parameter values may be quite expensive. In many cases, e.g., optimization, control, uncertainty quantification, and other settings, solutions are needed for many sets of parameter values. We consider the case of the time-dependent Navier-Stokes equations for which a recently developed ensemble-based method allows for the efficient determination of the multiple solutions corresponding to many parameter sets. The method uses the average of the multiple solutions at any time step to define a linear set of equations that determines the solutions at the next time step. To significantly further reduce the costs of determining multiple solutions of the Navier-Stokes equations, we incorporate a proper orthogonal decomposition (POD) reduced-order model into the ensemble-based method. The stability and convergence results for the ensemble-based method are extended to the ensemble-POD approach.  Numerical experiments are provided that illustrate the accuracy and efficiency of computations determined using the new approach.
\end{abstract}


\section{Introduction}
Computing an ensemble of solutions of fluid flow equations for a set of parameters or initial/boundary conditions for, e.g., quantifying uncertainty or sensitivity analyses or to make predictions, is a common procedure in many engineering and geophysical applications. One common problem faced in these calculations is the excessive cost in terms of both storage and computing time. 
Thanks to recent rapid advances in parallel computing as well as intensive research in ensemble-based data assimilation, it is now possible, in certain settings, to obtain reliable ensemble predictions using only a small set of realizations. 
Successful methods that are currently used to generate perturbations in initial conditions include the Bred-vector method, \cite{TK93}, the singular vector method, \cite{BP95}, and the ensemble transform Kalman filter, \cite{BEM01}. Despite all these efforts, the current level of available computing power is still insufficient to perform high-accuracy ensemble computations for applications that deal with large spatial scales such as numerical weather prediction. In such applications, spatial resolution is often sacrificed to reduce the total computational time. For these reasons the development of efficient methods that allow for fast calculation of flow ensembles at a sufficiently fine spatial resolution is of great practical interest and significance.

Only recently, a first step was taken in \cite{JL14,JL15} where a new algorithm was proposed for computing an ensemble of solutions of the time-dependent Navier-Stokes equations (NSE) with different initial condition and/or body forces. At each time step, the new method employs the same coefficient matrix for all ensemble members. This reduces the problem of solving multiple linear systems to solving one linear system with multiple right-hand sides. There have been many studies devoted to this type of linear algebra problem and efficient iterative methods have been developed to significantly save both storage and computing time, e.g., block CG \cite{FOP95}, block QMR \cite{FM97}, and block GMRES \cite{GS96}. Even for some direct methods, such as the simple LU factorization, one can save considerable computing cost. 

Because the main goal of the ensemble algorithm is computational efficiency, it is natural to consider using reduced-order modeling (ROM) techniques to further reduce the computational cost. Specifically, we consider the proper orthogonal decomposition (POD) method which has been extensively used in the engineering community since it was introduced in \cite{Lumley} to extract energetically coherent structures from turbulent velocity fields. POD provides an optimally ordered, orthonormal basis in the least-squares sense, for given sets of experimental or computational data. The reduced order model is then obtained by truncating the optimal basis. 

Research on POD and its application to the unsteady NSE has been and remains a highly active field. Recent works improving upon POD have dealt with the combination of Galerkin strategies with POD \cite{BCI13,CAMS13}, stabilization techniques \cite{ANR09,CIJS14,SK05}, and regularized/large eddy simulation POD models for turbulent flows \cite{ZIJT11,ZIJT12}.

In this paper, we study a Galerkin proper orthogonal decomposition (POD-G-ROM) based ensemble algorithm for approximating solutions of the NSE.  Accordingly, our aim in this paper is to develop and demonstrate a procedure for the rapid solution of multiple solutions of the NSE, requiring only the solution of one reduced linear system with multiple right-hand sides at each time step.

\subsection{Previous works on ensemble algorithms}

The ensemble method given in \cite{JL14} is first-order accurate in time and requires a CFL-like time step condition to ensure stability and convergence. Two ensemble eddy viscosity numerical regularizations are studied in \cite{JL15} to relax the time step restriction. These two methods utilized the available ensemble data to parametrize the eddy viscosity based on a direct calculation of the kinetic energy in fluctuations without further modeling. They both give the same parametrization for each ensemble member and thus preserve the efficiency of the ensemble algorithm. The extension of the ensemble method to higher-order accurate ensemble time discretization is nontrivial. For instance, the method is not extensible to the most commonly used Crank-Nicolson scheme. Making use of a special combination of a second-order in time backward difference formula and an explicit second-order Adams-Bashforth treatment of the nonlinear term, a second-order accurate in time ensemble method was developed in \cite{J15}.Another second-order ensemble method with improved accuracy is presented in \cite{J16}. The ensemble algorithm was further used in \cite{JKL15} to model turbulence. By analyzing the evolution of the model variance, it was proved that the proposed ensemble based turbulence model converges to statistical equilibrium, which is a desired property of turbulence models.

\section{Notation and preliminaries}
	
Let $\Omega \subset \mathbb{R}^{d}$, $d= 2,3$, denote an open regular domain with boundary $\partial\Omega$ and let $[0,T]$ denote a time interval. Consider $J$ Navier-Stokes equations on a bounded domain, each subject to the no-slip boundary condition, and driven by $J$ different initial conditions $u^{j,0}(x)$ and body force densities $f^{j}(x,t)$, i.e., for $j=1,\ldots,J$, we have
\begin{equation}\label{eq:NSE}
\left\{\begin{aligned}
u_{t}^j+u^{j}\cdot\nabla u^{j}-\nu\triangle u^{j}+\nabla p^{j}  &
=f^{j}(x,t)&\quad\forall x\in\Omega\times(0,T]\\
\nabla\cdot u^{j}  &  =0&\quad\forall x\in\Omega\times(0,T]\\
u^{j}  &  =0&\quad\forall x\in\partial\Omega\times(0,T]\\
u^{j}(x,0)  &  =u^{j,0}(x)&\quad\forall x\in\Omega,
\end{aligned}\right.
\end{equation}
where $\nu$ denotes the given constant kinematic viscosity of the fluid and $u^{j}(x,t)$ and $p^{j}(x,t)$ respectively denote the velocity and pressure of the fluid flow.

We denote by $\|\cdot\|$ and $(\cdot,\cdot)$ the $L^{2}(\Omega)$ norm and inner product, and denote by $\Vert \cdot \Vert_{L^p}$ and $\Vert \cdot \Vert_{W_p^k}$ the $L^p(\Omega)$ norms and the Sobolev $W_p^k(\Omega)$ norms respectively. The space $H^k(\Omega)$ is the Sobolev space $W_2^k(\Omega)$, equipped with norm $\Vert \cdot \Vert_k$. The space $H^{-1}(\Omega)$ denotes the dual space of bounded linear functionals defined on $H^{1}_{0}(\Omega)=\{v\in H^{1}(\Omega)\,:\,v=0 \mbox{ on } \partial\Omega\}$; this space is endowed with the norm
$$
\|f\|_{-1}=\sup_{0\neq v\in X}\frac{(f,v)}{\| \nabla v\| } 
\quad\forall f\in H^{-1}(\Omega).
$$

The solutions spaces $X$ for the velocity and $Q$ for the pressure are respectively defined as
$$
\begin{aligned}
X : =& [H^{1}_{0}(\Omega)]^{d} = \{ v \in [L^{2}(\Omega)]^{d} \,:\, \nabla v \in [L^{2}(\Omega)]^{d \times d} \ \text{and} \  v = 0 \ \text{on} \ \partial \Omega \} \\
Q : =& L^{2}_{0}(\Omega) = \Big\{ q \in L^{2}(\Omega) \,:\, \int_{\Omega} q dx = 0 \Big\}.
\end{aligned}
$$

Let $\Vert \cdot \Vert_0$ denote the usual $L^2$ norm. For a function $v(x,t)$ that is well defined on $\Omega \times [0,T]$ 
we define the norms
$$
\||v|\|_{2,s} : = \Big(\int_{0}^{T}\|v(\cdot,t)\|_{s}^{2}dt\Big)^{\frac{1}{2}}
\qquad \text{and} \qquad 
\||v|\|_{\infty,s} := \text{ess\,sup}_{[0,T]}\|v(\cdot,t)\|_{s} .
$$
The subspace of $X$ consisting of weakly divergence free functions is defined as
$$
V :=\{v\in X \,:\,(\nabla\cdot v,q)=0\,\,\forall q\in Q\} \subset X.
$$

A weak formulation of (\ref{eq:NSE}) is given as follows: for $j=1, \ldots, J$, find $u^j:(0,T]\rightarrow X$ and $p^j:(0,T]\rightarrow Q$ that, for almost all $t\in(0,T]$, satisfy 
\begin{equation}\label{wfwf}
\left\{\begin{aligned}
(u_{t}^j,v)+(u^{j}\cdot\nabla u^{j},v)+\nu(\nabla u^{j},\nabla v)-(p^{j}
,\nabla\cdot v)  &  =(f^{j},v)&\quad\forall v\in X\\
(\nabla\cdot u^{j},q)  &  =0&\quad\forall q\in Q\\
u^{j}(x,0)&=u^{j,0}(x).&
\end{aligned}\right.
\end{equation}

Conforming velocity and pressure finite element spaces based on a regular triangulation of $\Omega$ having maximum triangle diameter $h$ are respectively denoted by
$$
X_{h}\subset X\qquad\mbox{and}\qquad Q_{h}\subset Q.
$$
We assume that the pair of spaces $(X_h,Q_h)$ satisfy the discrete inf-sup (or $LBB_h$) condition required for stability of finite element approximation; we also assume that the finite element spaces satisfy the approximation properties
$$
\begin{aligned}
\inf_{v_h\in X_h}\| v- v_h \|&\leq C h^{s+1}&\forall v\in [H^{s+1}(\Omega)]^d\\
\inf_{v_h\in X_h}\| \nabla ( v- v_h )\|&\leq C h^s&\forall v\in [H^{s+1}(\Omega)]^d\\
\inf_{q_h\in Q_h}\|  q- q_h \|&\leq C h^s&\forall q\in H^{s}(\Omega)
\end{aligned}
$$
for a constant $C>0$ having value independent of $h$. The total number of finite element degrees of freedom is given by $\dim X_h + \dim Q_h$. A concrete example for which the $LBB_h$ stability condition approximation estimates are known to hold is the family of Taylor-Hood $P^s$-$P^{s-1}$, $s\geq 2$, element pairs \cite{GR79, Max89}. For the most commonly used $s=2$ Taylor-Hood element pair based on a tetrahedral grid, $\dim X_h + \dim Q_h$ is roughly equal to three times the number of vertices plus twice the number of edges.

Further, in this paper we will need to solve the NSE \eqref{eq:NSE} using a second order time stepping scheme (e.g., Crank Nicolson). We will assume the FE approximations satisfy the following error estimates:
\begin{gather}
\| u- u_h \|\leq C (h^{s+1}+\Delta t^{2} )\\
\| \nabla ( u- u_h ) \|\leq C (h^s+\Delta t^{2} ).
\end{gather}

The subspace of $X_{h}$ consisting of discretely divergence free functions is defined as
$$
V_{h} :=\{v_{h}\in X_{h}\,:\,(\nabla\cdot v_{h},q_{h})=0\,\,\forall
q_{h}\in Q_{h}\}  \subset X.
$$
Note that in most cases, and for the Taylor-Hood element pair in particular, $V_{h} \not\subset V$, i.e., discretly divergence free functions are not divergence free.

As is common to do, we define the explicitly skew-symmetric trilinear form introduced by Temam given by 
$$
b^{\ast}(w,u,v):=\frac{1}{2}(w\cdot\nabla u,v)-\frac{1}{2}(w\cdot\nabla v,u)
\qquad\forall u,v,w\in [H^1(\Omega)]^d.
$$
This form satisfies the bounds, \cite{Layton08}
\begin{gather}
b^{\ast}(w,u,v)\leq C \|  \nabla w\|   \| \nabla u\| (\|  v \|  \| \nabla
v \| )^{1/2}\qquad\forall u, v, w \in X, \label{In1}\\
b^{\ast}(w,u,v)\leq C (\| w \| \|  \nabla w\| )^{1/2}  \| \nabla u\|  \| \nabla
v \| \qquad\forall u, v, w \in X .\label{In2}
\end{gather}
Moreover, we have that $b^{\ast}(u,u,v)= (u\cdot\nabla u,v)$ for all $u\in V , v\in X$ so that we may replace the nonlinear term $(u^{j}\cdot\nabla u^{j},v)$ in \eqref{wfwf} by $b^{\ast}(u^j,u^j,v)$. The advantage garnered through the use of $b^{\ast}(w,u,v)$ compared to $(w\cdot\nabla u,v)$ is that $b^{\ast}(w,u,u)=0$ for all $u,w\in X$ whereas $(w\cdot\nabla u,v)=0$ only if $w\in V$.

\begin{definition}\label{def21}
Let $t^{n}=n\Delta t$, $n=0,1,2,\ldots,N$, where $N:=T/\Delta t$, denote a partition of the interval $[0,T]$. For $j=1, \ldots, J$ and $n=0,1,2,\ldots,N$, let $u^{j,n}(x):=u^{j}(x,t^{n})$. Then, the \text{\bf ensemble mean} is defined, for $n=0,1,2,\ldots,N$, by
$$
 <u>^n : =\frac{1}{J}\sum_{j=1}^{J}u^{j,n}.\label{Enmean}
$$
\end{definition}

For $j=1,\ldots,J$, let $u_{h}^{j,0}(x)\in X_h$ denote approximations, e.g., interpolants or projections, of the initial conditions $u^{j,0}(x)$. Then, the full space-time discretization of (\ref{eq:NSE}), or more precisely of \eqref{wfwf}, we consider is given as follows: {\em given, for $j=1,\ldots,J$, $u_{h}^{j,0}(x)\in X_h$ and $f^j(x,t)\in[H^{-1}(\Omega)]^d$ for almost every $t\in(0,T]$, find, for $n=0,1,\ldots,N-1$ and for $j=1,\ldots,J$, $u_{h}^{j,n+1}(x)\in X_{h}$ and $p_{h}^{j,n+1}(x)\in Q_{h}$ satisfying}
\begin{equation}\label{En-full-FE}\emph{}
\left\{\begin{aligned}
\big(\frac{u_{h}^{j, n+1}-u_{h}^{j,n}}{\Delta t},v_{h}\big)+b^{\ast}(<u_{h}
>^{n},u_{h}^{j, n+1},v_{h})+&b^{\ast}(u_{h}^{j,n}-<u_{h}>^{n},u_{h}^{j,n}
,v_{h})\\
-(p_{h}^{j, n+1},\nabla\cdot v_{h})+\nu(\nabla u_{h}^{j,n+1},\nabla
v_{h})&=(f^{j,n+1},v_{h}) \qquad\forall v_{h}\in X_{h}
\\
(\nabla\cdot u_{h}^{j,n+1},q_{h})&=0\qquad\qquad\qquad\forall q_{h}\in Q_{h}.
\end{aligned}\right.
\end{equation}
We refer to this discretization as {\em En-full-FE} indicating that we are referring to an ensemble-based discretization of \eqref{wfwf} using a high-dimensional finite element space. This ensemble-based discretization of the NSE is noteworthy because the system \eqref{En-full-FE} is not only {\em linear} in the unknown functions $u_{h}^{j,n+1}(x)$ and $p_{h}^{j,n+1}(x)$, but because of the use of ensembles, we also have that the coefficient matrix associated with \eqref{En-full-FE} is independent of $j$, i.e., at each time step, all members of the ensemble can be determined from $J$ linear algebraic systems all of which have the same coefficient matrix. On the other hand, the linear system can be very large because in practice $\dim X_h+\dim Q_h$ can be very large. This observation, in fact, motivates interest in building reduced-order discretizations of the NSE.  

Because $X_{h}$ and $Q_{h}$ are assumed to satisfy the $LBB_{h}$ condition, \eqref{En-full-FE} can be more compactly expressed as follows: {\em given, for $j=1,\ldots,J$, $u_{h}^{j,0}(x)\in X_h$ and $f^j(x,t)\in[H^{-1}(\Omega)]^d$ for almost every $t\in(0,T]$, find, for $n=0,1,\ldots,N-1$ and for $j=1,\ldots,J$,  $u_{h}^{j, n+1}(x)\in V_{h}$ satisfying}
\begin{equation}\label{eq: conv}
\begin{aligned}
\big(\frac{u_{h}^{j, n+1}-u_{h}^{j,n}}{\Delta t},v_{h}\big)+b^{\ast}(&<u_{h}
>^{n},u_{h}^{j,n+1},v_{h})+b^{\ast}(u_{h}^{j,n}-<u_{h}>^{n},u_{h}
^{j,n},v_{h})\\
&+\nu(\nabla u_{h}^{j,n+1},\nabla v_{h})=(f^{j,n+1},v_{h})
\qquad\forall v_{h}\in V_{h}.
\end{aligned}
\end{equation}
Note that in general it is a difficult matter to construct a basis for the space $V_h$ so that in practice, one still works with \eqref{En-full-FE}. We introduce the reduced system \eqref{eq: conv} so as to facilitate the analyses given in later sections.

\section{Proper orthogonal decomposition (POD) reduced-order modeling}
\label{PODsec}

The POD model reduction scheme can be split into two main stages: an offline portion and an online portion. In the offline portion, one collects into what is known as a snapshot set the solution of a partial differential equation (PDE), or more precisely, of a discrete approximation to that solution, for a number of different input functions and/or evaluated at several time instants. The snapshot set is hopefully generated in such a way that it is representative of the behavior of the exact solution. The snapshot set is then used to generate a POD basis, hopefully of much smaller cardinality compared to that of the full finite element space, that provides a good approximation to the data present in the snapshot set itself. In the online stage, the POD basis is used to generate approximate solutions of the PDE for other input functions; ideally these will be accurate approximations achieved much more cheaply compared to the use of a standard method such as a standard finite element method. 

In the rest of this section, we delve into further detail about the generation of the snapshot set, the construction of the POD basis in a finite element setting, and how the POD basis can be used to construct a reduced-order model for the NSE in the ensemble framework. This section will focus on the framework specific to this paper; for more detailed presentations about POD, see, e.g., \cite{CGS12,GPS07,HLB96,V01}.

\subsection{Snapshot set generation}\label{snapshot}

The offline portion of the algorithm begins with the construction of the snapshot set which consists of the solution of the PDE for a number of different input functions and/or evaluated at several different time instants. Given a positive integer $N_S$, let $0=t_0<t_1< \cdots < t_{N_S} = T$ denote a uniform partition of the time interval $[0,T]$. Note that this partition is usually much coarser than the partition of $[0,T]$ into $N$ intervals, introduced in Definition \ref{def21}, which is used to discretize the PDE, i.e., we have $N_S\ll N$. We first define the set of snapshots corresponding to exact solutions of the weak form of the NSE \eqref{wfwf}. For $j=1,\ldots,J_S$, we select $J_S$ different initial conditions $u^{j,0}(x)$ 
and denote by $u_S^{j,m}(x)\in X$ the exact velocity field satisfying \eqref{wfwf}, evaluated at $t=t_m$, $m=1,\ldots,N_S$, which corresponds to the initial condition $u^{j,0}(x)$. Then, the space spanned by the $J_S(N_S+1)$ so obtained snapshots is defined as
\begin{equation}\label{SnapSpace}
X_S:=\text{span} \{  u_S^{j,m}(x) \}_{j=1,m=0}^{J_S,N_S} \subset X.
\end{equation}
In the same manner, we can construct a set of snapshots $u_{h,S}^{j,m}(x)\in X_h$, $j=1,\ldots,J_S$, $m=0, 1,\ldots,N_S$, of finite element approximations of the velocity solution determined from a standard finite element discretization of \eqref{wfwf}. Note that one could also determine, at lesser cost but with some loss of accuracy, the snapshots from the ensemble-based discretization \eqref{En-full-FE}. We can then also define the space spanned by the $J_S (N_S+1)$ discrete snapshots as
\begin{equation}\label{SnapSpaceh}
X_{h,S}:=\text{span} \{  u_{h,S}^{j,m}(x) \}_{j=1,m=0}^{J_S,N_S} \subset V_h \subset X_h.
\end{equation}
Note that $S=\dim X_{h,S} \le J_S(N_S+1)$. The snapshots are finite element solutions so the span of the snapshots is a subset of the finite element space $X_h$. Additionally, it is important to note that by construction, the snapshots satisfy the discrete continuity equation so that the span of the snapshots is indeed a subspace of the discretly divergence free subspace $V_h\subset X_h$.

If we denote by $\vec{u}_S^{j,m}$ the vector of coefficients corresponding to the finite element function $u_{h,S}^{j,m}(x)$. With $K=\dim X_h$, we may also define the $K\times J_S(N_S+1)$ {\em snapshot matrix} $\mathbb{A}$ as
$$
\mathbb{A} = \big(\vec{u}_S^{1,0},\vec{u}_S^{1,1},  \ldots , \vec{u}_S^{1,N_S},  \vec{u}_S^{2,0},\vec{u}_S^{2,1},  \ldots , \vec{u}_S^{2,N_S}, \ldots , \vec{u}_S^{J_S,0}, \vec{u}_S^{J_S,1},  \ldots , \vec{u}_S^{J_S,N_S}\big),
$$
i.e., the columns of $\mathbb{A}$ are the finite element coefficient vectors of the discrete snapshots.

To construct a reduced basis that results in accurate approximations, the snapshot set must contain sufficient information about the dynamics of the solution of the PDE. In our context, this requires one to not only take a sufficient number of snapshots with respect to time, but also to select a set of initial conditions that generate a set of solutions that is representative of the possible dynamics one may encounter when using other initial conditions. In the POD framework for the NSE, the literature on selecting this set is limited. One of the few algorithms which has been explored in the ensemble framework is the previously mentioned Bred-vectors algorithm given in \cite{TK93}. Further exploration of this and other approaches for the selection of initial conditions is a subject for future research.

\subsection{Construction of the POD basis}

Using the set of discrete snapshots, we next construct the POD basis $\{{\varphi}_i(x)\}_{i=1}^R$. We define the POD function space $X_R$ as
$$
X_R :=\text{span}\{{\varphi}_i\}_{i=1}^R \subset X_{h,S} \subset V_h\subset X_h.
$$
There are a number of equivalent ways in which one may characterize the problem of determining $X_R$; for a full discussion see \cite[Section 2]{BGH}. For example, the POD basis construction problem can be defined as follows: determine an orthonormal basis $\{\varphi_i\}_{i=1}^S$ for $X_{h,S}$ such that for all $R\in \{1,\ldots,S\}$, $\{\varphi_i\}_{i=1}^R$ solves the following constrained minimization problem
\begin{equation}\label{Min}
\begin{aligned}
\min  
\sum_{k=1}^{J_S} \sum_{l=0}^{N_S}\Big \|  u_{h,s}^{k,l}-\sum_{j=1}^R (u_{h,s}^{k,l}, \varphi_j)\varphi_j\Big \| ^2 \\
\text{subject to } (\varphi_i, \varphi_j)= \delta_{ij}\quad\mbox{for $i,j=1,\ldots,R$},
\end{aligned}
\end{equation}
where $\delta_{ij}$ is the Kronecker delta and the minimization is with respect to all orthonormal bases for $X_{h,S}$. We note that by defining our basis in this manner we elect to view the snapshots  
as finite element functions as opposed to finite element coefficient vectors.   

Define the $J_S(N_S+1)\times J_S(N_S+1)$ correlation matrix $\mathbb{C} = \mathbb{A}^{T}\mathbb{M}\mathbb{A}$, where $\mathbb{M}$ denotes the Gram matrix corresponding to full finite element space. Then, the problem \eqref{Min} is equivalent to determine the $R$ dominant eigenpairs $\{\lambda_i,\vec{a}_i\}$ satifying
\begin{equation}
\label{eigProb}
\mathbb{C}\vec{a}_{i} = \lambda_{i}\vec{a}_{i}, \ \ \ |\vec{a}_{i}| = 1, \ \ \ \vec{a}^{T}_{i}\vec{a}_{j} = 0 \ \text{if} \ i \neq j, \ \text{and} \ \lambda_{i} \geq \lambda_{i-1} > 0,
\end{equation}
where $|\cdot|$ denotes the Euclidean norm of a vector. The finite element coefficient vectors corresponding to the POD basis functions are then given by
\begin{gather}
\vec\varphi_i = \frac{1}{\sqrt{\lambda_i}}\mathbb{A}\vec{a}_{i}, \ \ \ i = 1, \ldots, R.
\end{gather}
Alternatively, we can let $\mathbb{M} = \mathbb{S}^{T}\mathbb{S}$, and define $\widetilde{\mathbb{A}} = \mathbb{S}\mathbb{A}$ so that $\mathbb{C} = \mathbb{A}^{T}\mathbb{M}\mathbb{A} = \widetilde{\mathbb{A}}^{T}\widetilde{\mathbb{A}}$ and then determine the singular value decomposition of the modified snapshot matrix $\widetilde{\mathbb{A}}$; the vectors $\vec{a}_{i}$, $i = 1, \ldots, R$ are then given as the first $R$ left singular vectors of $\widetilde{\mathbb{A}}$ which correspond to the first $R$ singular values $\sigma_i=\sqrt{\lambda_i}$.

\subsection{POD reduced-order modeling}  

We next illustrate how a POD basis is used to construct a reduced-order model for the NSE within the ensemble framework. 
The discretized system that defines the POD approximation mimics that for the full finite element approximation, except that now we seek an approximation in the POD space $X_R$ having the basis $\{{\varphi}_i\}_{i=1}^R$. Specifically, for $j=1,\ldots,J$, we define the POD approximate initial conditions as $u_{R}^{j,0}(x)=\sum_{i=1}^R (u^{j,0}, {\varphi}_i)\varphi_i(x)\in X_R$ and then pose the following problem: {\em given $u_{R}^{j,0}(x)\in X_R$, for $n=0,1,\ldots,N-1$ and for $j=1,\ldots,J$, find $u_{R}^{j,n+1}\in X_R$ satisfying}
\begin{equation}\label{En-POD-Weak}
\begin{aligned}
\big(\frac{u_{R}^{j,n+1}-u_{R}^{j,n}}{\Delta t}, \varphi\big)+&b^{\ast}(<u_{R}
>^{n},u_{R}^{j,n+1},\varphi)+b^{\ast}(u_{R}^{j, n}-<u_{R}
>^{n},u_{R}^{j,n}
,\varphi)\\&
+\nu(\nabla u_{R}^{j,n+1},\nabla
\varphi)=(f^{j,n+1},\varphi)\qquad\forall \varphi\in X_{R}.
\end{aligned}
\end{equation}
We refer to this discretization as {\em En-POD} indicating that we are referring to an
ensemble-based discretization of (2.2) using a low-dimensional POD space. Note that because $X_R\subset V_h$, i.e., the POD approximation is by construction discretely divergence free, the pressure term in the POD-discretized NSE \eqref{En-POD-Weak} drops out and we are left with a system involving only the POD approximation to the velocity.  One further point of emphasis is that the $J$ initial conditions used in \eqref{En-POD-Weak} are different from the $J_S$ initial conditions used to construct the snapshot set, i.e., we use $J_S$ initial conditions to solve the full finite element system \eqref{En-full-FE} to determine the snapshots, and now solve $J$ additional approximations of the NSE by solving the much smaller POD system \eqref{En-POD-Weak}. 

As was the case for \eqref{En-full-FE}, the POD system \eqref{En-POD-Weak} is linear in the unknown $u_{R}^{j,n+1}$ and the associated coefficient matrix does not depend on $j$, i.e., it is the same for all realizations of the initial condition. On the other hand, \eqref{En-POD-Weak} is a system of $R$ equations in $R$ unknowns whereas \eqref{En-full-FE} involves $\dim X_h+\dim Q_h$ equations in the same number of unknowns, where $R$ and $\dim X_h+\dim Q_h$ denote the total number of POD and finite element degrees of freedom, respectively. Thus, if $R\ll \dim X_h+\dim Q_h$, {\em solving} \eqref{En-POD-Weak} requires much less cost compared to solving \eqref{En-full-FE}. In this way the offline cost of constructing the POD basis can be amortized over many online solves using the much smaller POD system. We address the {\em assembly} costs related to \eqref{En-POD-Weak} in Section \ref{numexp}.

\section{Stability analysis of En-POD}

We prove the conditional, nonlinear, long-time stability of solutions of (\ref{En-POD-Weak}).

The $L^2(\Omega)$ projection operator $\Pi_R$: $L^2(\Omega) \rightarrow X_R$ is defined by
\begin{equation}\label{eq:def_proj}
 (u-\Pi_R u , \varphi)=0\qquad \forall \varphi \in X_R.
\end{equation}
Denote by $\|\hspace{-1pt}|  \cdot \|\hspace{-1pt}|_2 $ the spectral norm for symmetric matrices and let ${\mathbb M}_R$ denote the $R\times R$ POD mass matrix with entries $[{\mathbb M}_R]_{i,i'}= (\varphi_i, \varphi_i')$ and ${\mathbb S}_R$ denote the $R\times R$ matrix with entries $[{\mathbb S}_R]_{i,i'}=[{\mathbb M}_R]_{i,i'}+\nu(\nabla \varphi_i, \nabla \varphi_{i'})$, $i,i'=1,\ldots,R$. It is shown in \cite{KV01} that
\begin{equation}\label{lm:inverse}
\|  \nabla \varphi \|  \leq \big(\|\hspace{-1pt}|  {\mathbb S}_R \|\hspace{-1pt}| _2 \|\hspace{-1pt}|  {\mathbb M}_R^{-1} \|\hspace{-1pt}| _2\big)^{1/2} \| \varphi \| \qquad \forall \varphi \in X_R.
\end{equation}
As $X_R\subset X_h$, we have the following lemma, see of \cite[page 276]{JL14} for proof.
\begin{lemma}\label{lm:skew} 
For any $u_{R}, v_{R}, w_{R} \in X_{R}$,
\begin{align*}
b^{\ast}(u_{R},v_{R},w_{R})=\int_{\Omega}u_{R}\cdot\nabla v_{R} \cdot w_{R}
\text{ }dx + \frac{1}{2}\int_{\Omega}(\nabla\cdot u_{R})(v_{R}\cdot w_{R}
)\text{ }dx.
\end{align*}
\end{lemma}

\begin{theorem}\label{th:En-POD}
{\rm[Stability of En-POD]} For $n=0,\ldots,N-1$ and $j=1,\ldots,J$, let $u_{R}^{j,n+1}$ satisfy
\eqref{En-POD-Weak}. Suppose the time-step condition
\begin{equation}
\big(C{\nu}^{-1} \|\hspace{-1pt}|  {\mathbb S}_R \|\hspace{-1pt}| _2^{1/2}\| \nabla (u_{R}^{j,n}-<u_R>^n)\| ^{2}
\big)\Delta t
\leq1\qquad \mbox{for $j=1,\ldots,J$} \label{ineq:CFL-h}
\end{equation}
holds. Then, for
$n=1,\ldots,N$,
\begin{equation}\label{Stab:result}
\begin{aligned}
\frac{1}{2}\|&u_{R}^{j,n}\|^{2}+\frac{1}{4}\sum_{n'=0}^{n-1}\|u_{R}
^{j, n'+1}-u_{R}^{j, n'}\|^{2}+\frac{\nu\Delta t}{4}\|\nabla u_{R}^{j,n}\|^{2}
+\frac{\nu\Delta t}{4} \sum_{n'=0}^{n-1}\|\nabla u_{R}^{j, n'+1}\|^{2}\\
&\leq\sum_{n'=0}^{n-1}\frac{\Delta t}{2\nu}\|f^{j,n'+1}\|_{-1}^{2}+ \frac{1}{2}\|u_{R}^{j,0}\|^{2}+\frac{\nu\Delta t}{4}\|\nabla u_{R}^{j,0}\|^{2}\qquad \mbox{for $j=1,\ldots,J$} .
\end{aligned}
\end{equation}
\end{theorem}

\begin{proof}
The proof is provided in Appendix \ref{app1}.
\end{proof}

\begin{remark}
In the time-step condition, the constant C is dependent on the shape of the domain and the mesh as a result of the use of inverse inequality in the proof. For a fixed mesh on a fixed domain, C is a generic constant that is independent of the time step $\Delta t$, the solution $u^j$ and viscosity $\nu$.
\end{remark}

\section{Error analysis of En-POD}

We next provide an error analysis for En-POD solutions.

\begin{lemma} \label{lm:L2err} 
{\rm[$L^2(\Omega)$ norm of the error between snapshots and their projections onto the POD space]}  We have
\begin{equation}\label{errL2_1}
\frac{1}{J_S(N_S+1)} \sum_{j=1}^{J_S} \sum_{m=0}^{N_S}\Big \|  u_{h,S}^{j,m}-\sum_{i=1}^R (u_{h,S}^{j,m}, \varphi_i)\varphi_i\Big \| ^2 = \sum_{i=R+1}^{J_S(N_S+1)} \lambda_i 
\end{equation}
and thus for $j=1,\ldots,J_S$,
\begin{equation}\label{errL2_2}
\frac{1}{N_S+1} \sum_{m=0}^{N_S}\Big \|  u_{h,S}^{j,m}-\sum_{i=1}^R (u_{h,S}^{j,m}, \varphi_i)\varphi_i\Big \| ^2 \leq J_S\sum_{i=R+1}^{J_S(N_S+1)} \lambda_i .
\end{equation}
\end{lemma}
\begin{proof}
The proof of \eqref{errL2_1} follows exactly the proof of \cite[Theorem 3]{V01}; \eqref{errL2_2} is then a direct consequence of \eqref{errL2_1}.
\end{proof}

\begin{lemma}\label{lm:H1err}{\rm [$H^1(\Omega)$ norm of the error between snapshots and their projections in the POD space]}  We have
\begin{equation}\label{errH1_1}
\frac{1}{J_S(N_S+1)} \sum_{j=1}^{J_S} \sum_{m=0}^{N_S}\Big \| \nabla \Big( u_{h,S}^{j,m}-\sum_{i=1}^R (u_{h,S}^{j,m}, \varphi_i)\varphi_i\Big)\Big\| ^2 
=\sum_{i=R+1}^{J_S(N_S+1)} \lambda_i  \|  \nabla \varphi_i\|^2
\end{equation}
and thus for $j=1,\ldots,J_S$,
\begin{equation}\label{errH1_2}
\frac{1}{N_S+1} \sum_{m=0}^{N_S} \Big\| \nabla\Big( u_{h,S}^{j,m}-\sum_{i=1}^R (u_{h,S}^{j,m}, \varphi_i)\varphi_i \Big)\Big\| ^2 \leq J_S\sum_{i=R+1}^{J_S(N_S+1)} \lambda_i  \|  \nabla \varphi_i\|^2.
\end{equation}
\end{lemma}

\begin{proof}
The proof of \eqref{errH1_1} follows exactly the proof of \cite[Lemma 3.2]{IW14}; \eqref{errH1_2} is then direct consequence of \eqref{errH1_1}.
\end{proof}

\begin{lemma}\label{lm:Projerr}{\rm [Error in the projection onto the POD space]}
Consider the partition $0=t_0<t_1< \cdots < t_{N_S} = T$ used in Section \ref{snapshot}. 
For any $u \in H^{1}(0,T;[H^{s+1}(\Omega)]^d)$, let $u^{m}=u(\cdot, t_m)$. Then, the error in the projection onto the POD space $X_R$ satisfies the estimates
\begin{equation}\label{pel2}
\begin{aligned}
\frac{1}{N_S+1}& \sum_{m=0}^{N_S} \|  u^{m}-\Pi_R u^m\| ^2
\\&
\leq \inf_{j\in\{1,\ldots,J_S\}}\frac{2}{N_S+1} \sum_{m=0}^{N_S} \|  u^{m}-u^{j,m}_S\| ^2+C\left( h^{2s+2} + \triangle t^4  \right)+ 2J_S\sum_{i=R+1}^{J_S(N_S+1)} \lambda_i 
\end{aligned}
\end{equation}
\begin{equation}\label{peh1}
\begin{aligned}
\frac{1}{N_S+1} &\sum_{m=0}^{N_S} \|  \nabla \left(u^{m}-\Pi_R u^m\right)\| ^2
\\&
\leq \inf_{{j\in\{1,\ldots,J_S\}}}\frac{2}{N_S+1} \sum_{m=0}^{N_S} \left(\|  \nabla (u^{m}-u_S^{j,m})\| ^2+\|\hspace{-1pt}|  {\mathbb S}_R \|\hspace{-1pt}| _2 \|   u^{m}-u_S^{j,m}\| ^2  \right)
\\&
+(C+h^2 \|\hspace{-1pt}|  {\mathbb S}_R \|\hspace{-1pt}| _2 ) h^{2s} + (C+\|\hspace{-1pt}|  {\mathbb S}_R \|\hspace{-1pt}| _2 )\triangle t^4  
+ 2J_S\sum_{i=R+1}^{J_S(N_S+1)} \|  \nabla \varphi_i\| ^2\lambda_i . 
\end{aligned}
\end{equation}
\end{lemma}

\begin{proof}
The proof is provided in Appendix \ref{app2}.
\end{proof}

To bound the error between the POD based approximations and the true solutions, we assume
the following regularity for the true solutions and body forces:
\begin{gather*}
u^{j} \in L^{\infty}(0,T;H^{s+1}(\Omega))\cap H^{1}(0,T;H^{s+1}(\Omega))\cap
H^{2}(0,T;L^{2}(\Omega)),\\
p^{j} \in L^{2}(0,T;H^{s}(\Omega)),\quad \text{and}\quad f^{j} \in L^{2}
(0,T;L^{2}(\Omega)).
\end{gather*}
{
We assume the following estimate is also valid as done in \cite{IW14}.
\begin{assumption}\label{assumption1}
Consider the partition $0=t_0<t_1< \cdots < t_{N_S} = T$ used in Section \ref{snapshot}. 
For any $u \in H^{1}(0,T;[H^{s+1}(\Omega)]^d)$, let $u^{m}=u(\cdot, t_m)$. Then, the error in the projection onto the POD space $X_R$ satisfies the estimates

\begin{equation}\label{peh2}
\begin{aligned}
& \|  \nabla \left(u^{m}-\Pi_R u^m\right)\| ^2
\\&
\leq \inf_{{j\in\{1,\ldots,J_S\}}}\frac{2}{N_S+1} \sum_{m=0}^{N_S} \left(\|  \nabla (u^{m}-u_S^{j,m})\| ^2+\|\hspace{-1pt}|  {\mathbb S}_R \|\hspace{-1pt}| _2 \|   u^{m}-u_S^{j,m}\| ^2  \right)
\\&
+(C+h^2 \|\hspace{-1pt}|  {\mathbb S}_R \|\hspace{-1pt}| _2 ) h^{2s} + (C+\|\hspace{-1pt}|  {\mathbb S}_R \|\hspace{-1pt}| _2 )\triangle t^4  
+ 2J_S\sum_{i=R+1}^{J_S(N_S+1)} \|  \nabla \varphi_i\| ^2\lambda_i . 
\end{aligned}
\end{equation}
\end{assumption}
}
Let $e^{j,n}=u^{j,n}-u_{R}^{j,n}$ be the error between the true solution
and the POD approximation, then we have the following error estimates.

\begin{theorem}
[Error analysis of En-POD]\label{th:errEn-POD} Consider the
method (\ref{eq: conv}) and the partition $0=t_0<t_1< \cdots <t_{N_S}=T $ used in Section 3.1. Suppose that for any $0\leq n\leq N_S$, the following conditions hold
\begin{gather}
 \frac{C\triangle t \|\hspace{-1pt}|  {\mathbb S}_R \|\hspace{-1pt}| _2^{1/2}}{\nu}\| \nabla
(u_{R}^{ j,n}-<u_R>^n)\| ^{2} <1\text{ , \qquad} j=1,...,J. \label{cond:err}
\end{gather}
Then, for any $1\leq N\leq N_S$, there is a positive constant $C$ such that
\begin{equation}\label{ineq:err00}
\begin{aligned}
\frac{1}{2}\|  e^{j,N}\| ^{2}&+C\Delta t\sum_{n=0}^{N-1}\nu\| \nabla e^{j,n+1}\| ^{2}\\
&\leq C\Bigg ( \Delta t^2+h^{2s} +\Delta t \|\hspace{-1pt}|  {\mathbb S}_R \|\hspace{-1pt}| _2^{-1/2}
+ \|\hspace{-1pt}|  {\mathbb S}_R \|\hspace{-1pt}| _2 \Delta t^4+ \|\hspace{-1pt}|  {\mathbb S}_R \|\hspace{-1pt}| _2 h^{2s+2}\\
&+ \|\hspace{-1pt}|  {\mathbb S}_R \|\hspace{-1pt}| _2^{-1/2} h^{2s}\Delta t^{-1}
+ \|\hspace{-1pt}|  {\mathbb S}_R \|\hspace{-1pt}| _2^{1/2} h^{2s+2}\Delta t^{-1}
+  \|\hspace{-1pt}|  {\mathbb S}_R \|\hspace{-1pt}| _2^{1/2} \Delta t^3\\
& +(1+ N_S\ \|\hspace{-1pt}|  {\mathbb S}_R \|\hspace{-1pt}| _2^{-1/2}) \Big(\inf_{{j\in\{1,\ldots,J_S\}}}\frac{1}{N_S} \sum_{m=1}^{N_S} (\|  \nabla (u^{m}-u_S^{j,m})\| ^2\\
&\quad+\|\hspace{-1pt}|  {\mathbb S}_R \|\hspace{-1pt}| _2 \|   u^{m}-u_S^{j,m}\| ^2  )+ J_S\sum_{i=R+1}^{J_SN_S} \|  \nabla \varphi_i\| ^2\lambda_i  \Big)\Bigg)
\end{aligned}
\end{equation}
\end{theorem}

\begin{proof}
The proof is provided in Appendix \ref{app3}.
\end{proof}

\section{Numerical simulations}\label{numexp}

We investigate the efficacy of our algorithm via the numerical simulation of a flow between two offset circles \cite{JL14}. Before we discuss the examples and the numerical results, we briefly discuss the computational costs associated with the \emph{En-POD} algorithm and how they compare to those of the \emph{En-full-FE} algorithm.

\subsection{Computational costs}

As stated in Section \ref{PODsec}, we can split the computational cost of our algorithm into offline and online portions. In the offline portion, we generate the snapshot matrix $\mathbb{A}$ by solving the Navier-Stokes equations for $J_S$ perturbations. Using $\mathbb{A}$, we then generate a  reduced basis to be used in our online calculations. It is fair to assume that the cost of creating the snapshot matrix will dominate the cost of generating the reduced basis associated with the eigenvalue problem \eqref{eigProb}, especially when we consider that there exist very efficient techniques \cite{H90} for determining the partial SVD of matrices. 

Turning to the cost of solving the Navier-Stokes equation, the discrete systems that arise from a FEM discretization have been studied at great length. Whereas it is possible to use a nonlinear solver such as Newton's method or a nonlinear multigrid iteration, these methods often suffer from a lack of robustness. Instead, it is more popular to linearize the system and then to use the Schur complement approach. This allows for the use of a linear multigrid solver or Krylov method such as GMRES to solve the problem. For full details, see, e.g., \cite{R00}. Unfortunately, there are a number of factors such as the mesh size, the value of the Reynolds number, and the choice of pre-conditioner which make it very difficult to precisely estimate how quickly these methods converge.  

Estimating the online cost of the {\em En-POD} method, however, is much easier. Because the POD discrete system is small and dense and the ensemble method has $J$ right-hand sides, the most efficient way to solve this problem is, at each time step, to do a single LU factorization and a backsolve for each right-hand side. Denoting again by $R$ the cardinality of the reduced basis, the online cost of the {\em En-POD} method is
\begin{gather}
RB_{online} = N {\mathcal O}(R^{3}) + N J {\mathcal O}(R^{2}).
\end{gather}
We note that this process is highly parallelizable. For example, if we have access to $J$ total processors, then we can remove the factor $J$ in the second term.

It is important to note that the {\em assembly} of the low-dimensional reduced basis system requires manipulations involving the reduced basis which, as we have seen, are finite element functions so that, in general, that assembly involves computational costs that depend on the dimension of the finite element space. Thus, naive implementations of a reduced basis method involve assembly costs that are substantially greater than solving costs and which, given the availability of very efficient solvers, do not result in significant savings compared to that incurred by the full finite element discretization. For linear problems the stiffness matrix is independent of the solution so that one can assemble the small reduced basis stiffness matrix during the offline stage. For nonlinear problems, the discrete system changes at each time step (and generally at each interrogation of a nonlinear solver) so that, in general, it is not an easy matter to avoid the high assembly costs. However, because the nonlinearity in the Navier-Stokes system is quadratic, the assembly costs can again be shifted to the offline stage during which one assembles a low-dimensional third-order tensor that can be reused throughout the calculations.

Turning to the computational cost for the FEM ensemble method, as mentioned previously, the most efficient way to solve the resulting systems is a block solver (e.g., block GMRES). In trying to estimate the computational cost, we run into the same problem as we do for estimating the cost of solving the standard FEM discretization of the Navier-Stokes problem; specifically, it is very difficult to precisely determine how quickly any block solver converges.

Due to the difficulties outlined above in a priori estimation of the computational costs for both our algorithms we omit any CPU time comparison in the numerical experiments. Instead, we focus on the accuracy of our {\em En-POD} method, demonstrating that it is possible to achieve similar results as those given by the {\em En-full-FE} method. A more rigorous and thorough analysis comparing the computational cost of the {\em En-POD} and {\em En-full-FE} method is a subject of future research.

\subsection{Flow between two offset circles}\label{twoc}

For the numerical experiment we examine the two-dimensional flow between two offset circles with viscosity coefficient $\nu = \frac{1}{200}$. Specifically, the domain is a disk with a smaller offset disc inside. Let $r_{1}=1$, $r_{2}=0.1$, $c_{1}=1/2$, and $c_{2}=0$; then the domain is given by
\[
\Omega=\{(x,y):x^{2}+y^{2}\leq r_{1}^{2} \text{ and } (x-c_{1})^{2}
+(y-c_{2})^{2}\geq r_{2}^{2}\}.
\]

\noindent No-slip, no-penetration boundary conditions are imposed on both circles. All computations are done using the FEniCS software suite \cite{LNW12}. The deterministic flow driven by the counterclockwise rotational body force
\[
f(x,y,t)=\big(-4y(1-x^{2}-y^{2})\,,\,4x(1-x^{2}-y^{2})\big)^{T}
\]
displays interesting structures interacting with the inner circle. A K\'arm\'an vortex street is formed which then re-interacts with the inner circle and with itself, generating complex flow patterns. 

For our test problems, we generate perturbed initial conditions by solving a steady Stokes problem with perturbed body forces given by 
\[
f_{\epsilon}(x,y,t)=f(x,y,t)+\epsilon\big(\sin(3\pi x)\sin(3\pi y),\cos(3\pi
x)\cos(3\pi y)\big)^{T}
\]
with different perturbations defined by varying $\epsilon$. We discretize in space via the $P^2$-$P^1$ Taylor-Hood element pair. Meshes were generated using the FEniCS built-in \textbf{mshr} package with varying refinement levels. An example mesh is given in Figure \ref{meshEx1}.

\begin{figure}[h!]
\centering
\includegraphics{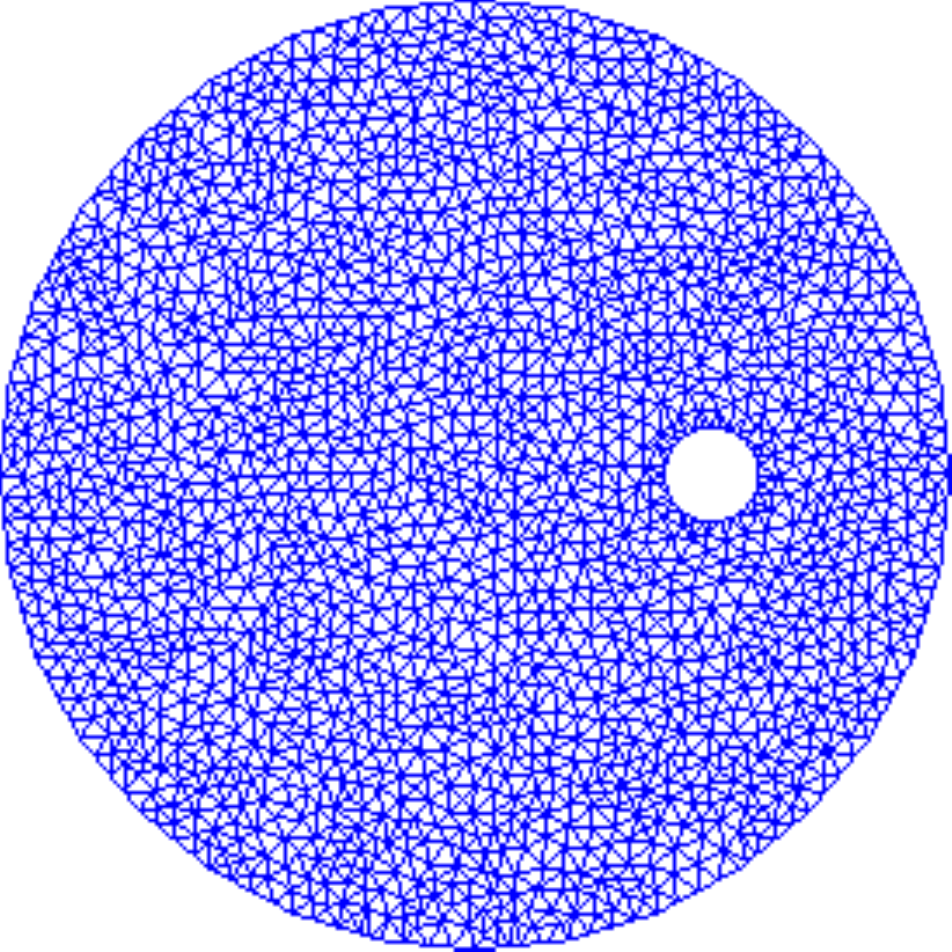} 
\caption{Mesh for flow between offset circles resulting in 16,457 total degrees of freedom for the Taylor-Hood element pair.}
\label{meshEx1}
\end{figure}

In order to generate the POD basis, we use two perturbations of the initial conditions corresponding to $\epsilon_{1} = 10^{-3}$ and $\epsilon_{2} = -10^{-3}$. Using a mesh that results in 16,457 total degrees of freedom and a fixed time step $\Delta t = .025$, we run a standard full finite element code\footnote{We also generated snapshots using the {\em En-full-FE} method. We found that we obtained exactly the same results as those reported here if instead we use a standard finite element method.} for each perturbation from $t_{0} = 0$ to $T=5$. For the time discretization we use the Crank-Nicolson method and take snapshots every $0.1$ seconds. In Figure \ref{eigvaldecay}, we illustrate the decay of the singular values generated the snapshot matrix.

\begin{figure}[h!]
\centering
\includegraphics[width=8cm]{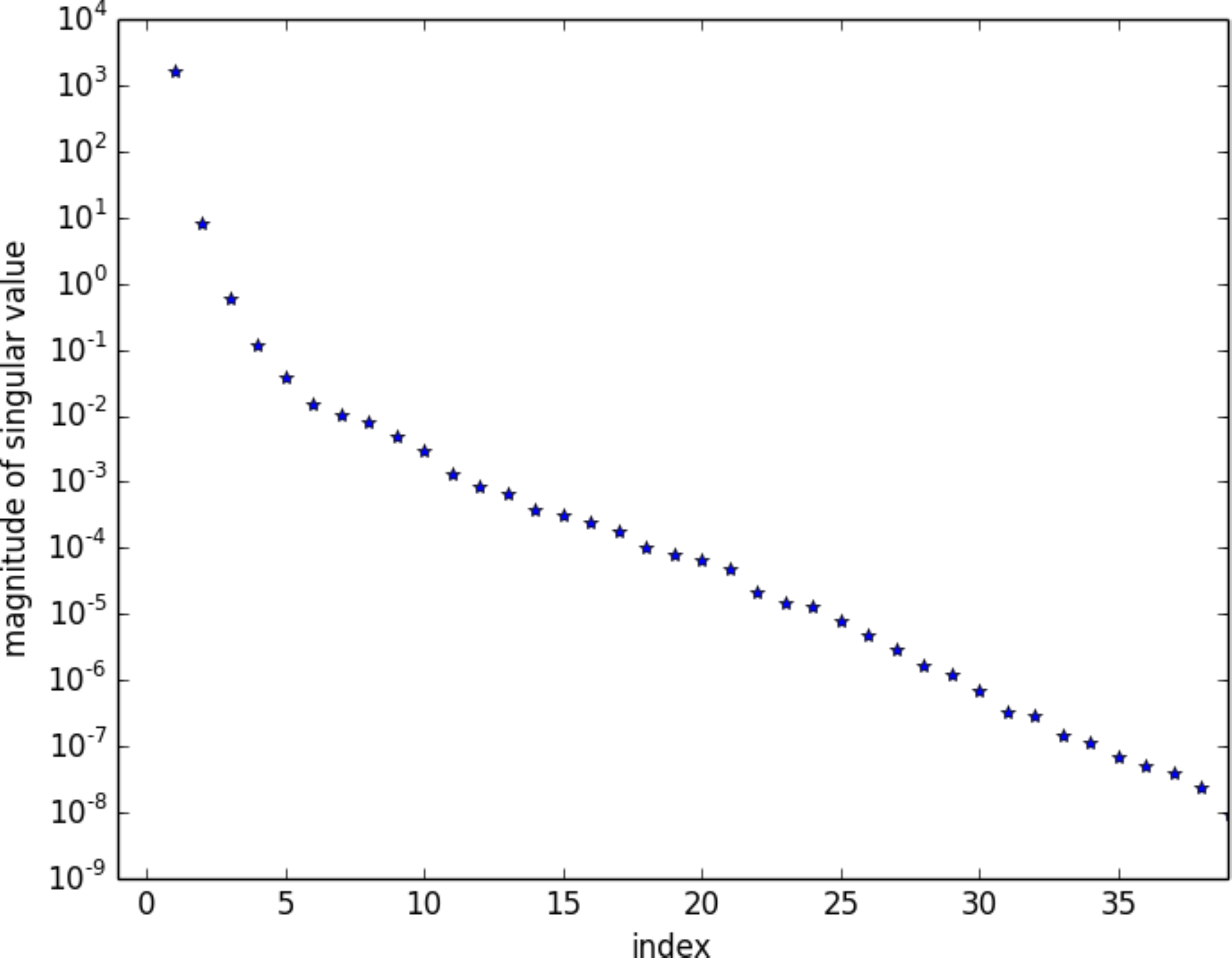} 
\caption{The 40 largest singular values of the snapshot matrix.}
\label{eigvaldecay}
\end{figure}

\subsection{Example 1}\label{datam}

The purpose of this example is to illustrate our theoretical error estimates and to show the efficacy of our method in a ``data mining'' setting, i.e., to show that we can accurately represent the information contained in the {\em En-full-FE} approximation which requires the specification of 16,457 coefficients by the {\em En-POD} approximation that requires the specification of a much smaller number of coefficients, in fact, merely 10 will do. Thus, we determine the {\em En-POD} approximation using the same perturbations, mesh, and time step as were used in the generation of the POD basis. We verify at each time step that condition \eqref{ineq:CFL-h} is satisfied. In order to illustrate the accuracy of our approach, we provide, in Figure \ref{noRomEx1}, plots of the velocity field of the ensemble average at the final time $T=5$ for both the {\em En-full-FE} and {\em En-POD} approximations. We also provide in Figure \ref{errorDiff} (left) the difference between the two ensemble averages at the final time $T = 5$. In addition, in Figure \ref{Energy_ex1}, we plot, for $0 \leq t \leq 5$ and for both methods, the energy $\frac{1}{2}\|u\|^{2}$ and the enstrophy $  \frac{1}{2}\nu\|\nabla \times u\|^{2}$.  

\begin{figure}[h!]
\centering
\includegraphics[height=5.5cm]{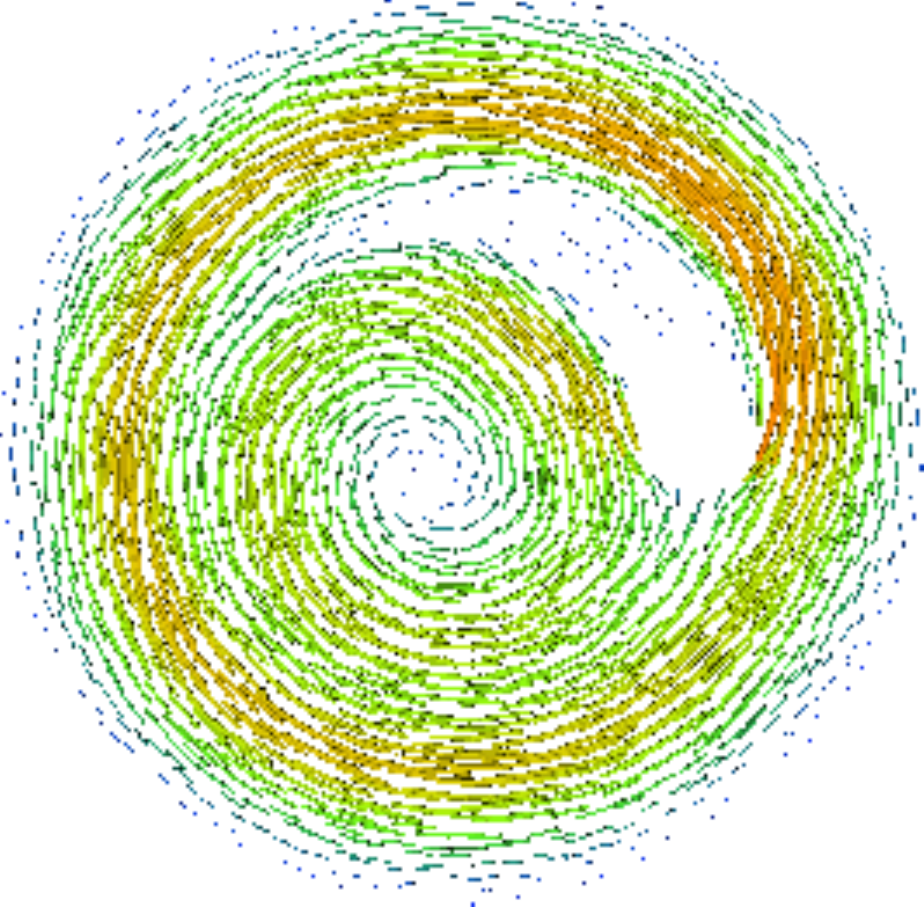}
\includegraphics[height=3cm]{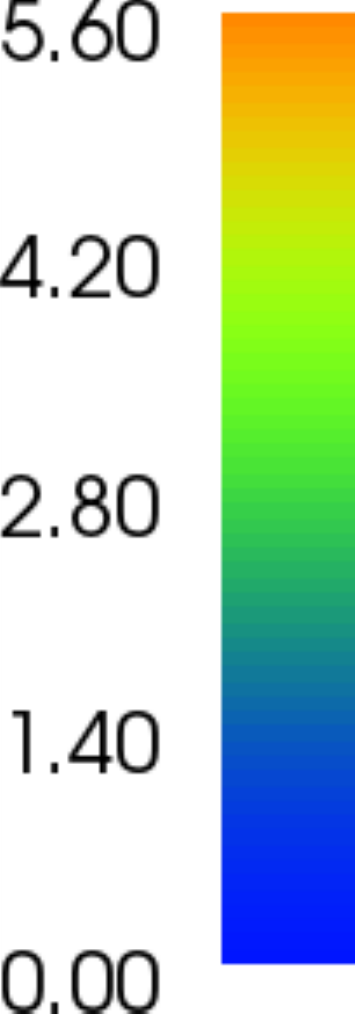} 
\includegraphics[height=5.5cm]{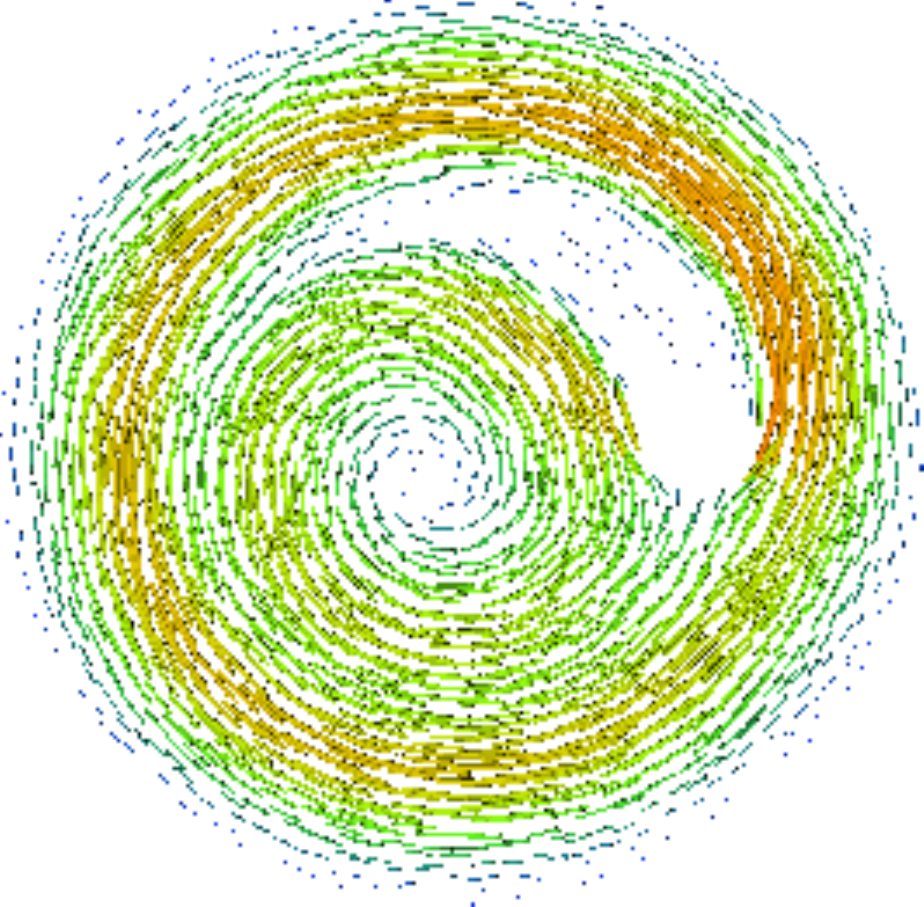} 
\caption{For Example 1, the ensemble average of the velocity field at the final time $T = 5$ of the En-full-FE (left) approximation and the En-POD approximation with $10$ reduced basis vectors (right).}
\label{noRomEx1}
\end{figure}

\begin{figure}[h!]
\centering
\includegraphics[height=5.5cm]{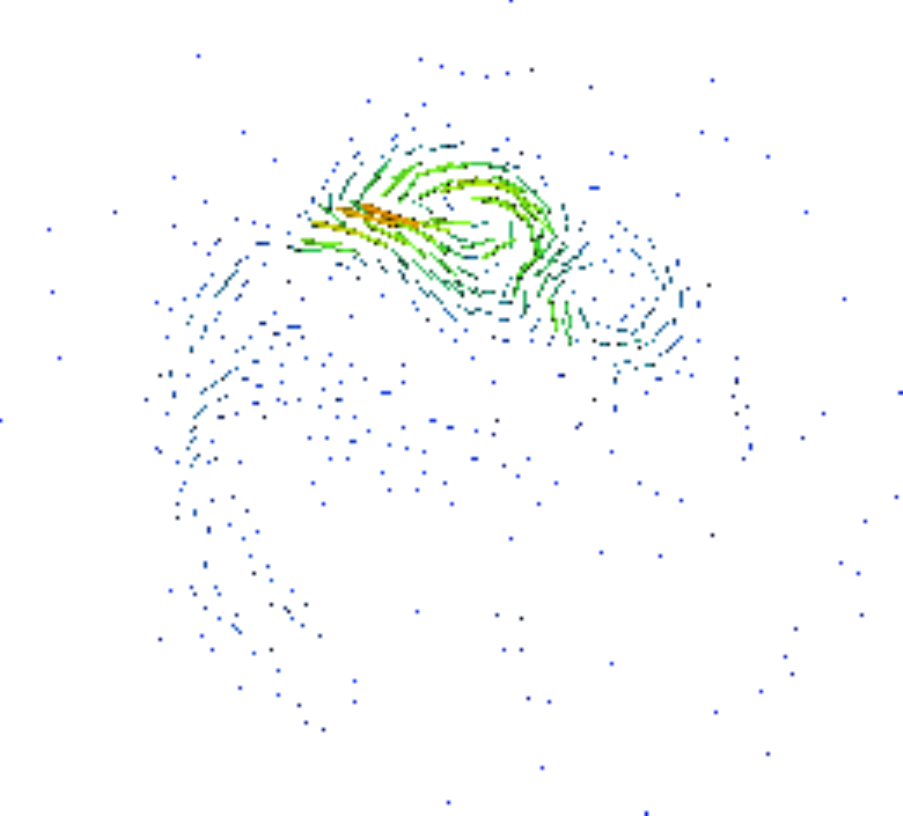}
\includegraphics[height=3cm]{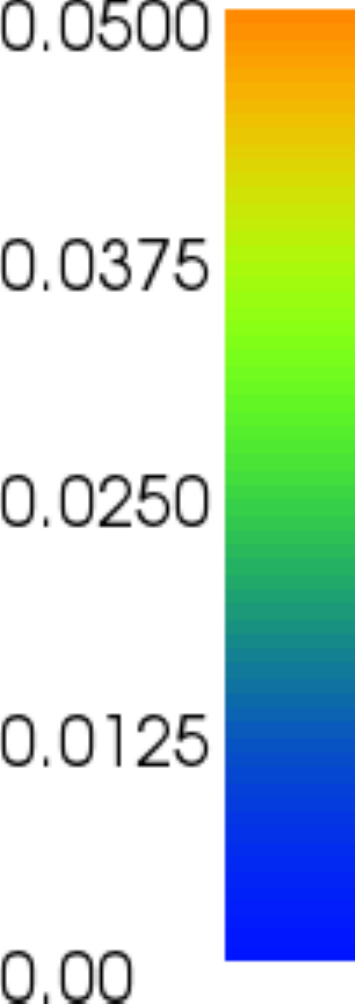} 
\caption{For Example 1, The difference between the ensemble average of the velocity field at the final time $T = 5$ of the En-full-FE (approximation and the En-POD approximation with $10$ reduced basis vectors.}
\label{errorDiff}
\end{figure}

We need $10$ POD basis functions to reproduce the flow with a reasonable level of accuracy. This is seen in Table \ref{tabEx1}(a) which shows a small discrete $L^{2}$ error corresponding to 10 basis vectors and, as the number of basis vectors increases beyond that, the error appears to decreases monotonically. Visual confirmation is given by comparing the two plots in Figure \ref{noRomEx1} as well as Figure \ref{errorDiff} (left); at time $T = 5$ the {\em En-POD} method appears to produce a flow which is very similar to that for the {\em En-full method}. Additionally, in Figure \ref{Energy_ex1}, we plot the energy and enstrophy of {\em En-POD} with varying cardinalities for the POD basis and for the  {\em En-full-FE} method. It can be seen that as the number of POD basis vectors increases our approximation improves with the {\em En-POD} energy and enstrophy becoming indistinguishable from that for the {\em En-full-FE}  for $10$ or more POD basis functions.

 \begin{table}[h!]
   \caption{F for Examples 1 and 2, the $L^{2}$ relative error $||u_{h}^{ave}- u_{R}^{ave}||_{2,0}$ vs. the dimension $R$ of the POD approximation.}\label{tabEx1}
\begin{center}
  \begin{tabular}{|c|c|c|c|c|c|c|c|}
     \multicolumn{2}{c}{\footnotesize(a) \em Example 1} &\multicolumn{1}{c}{}&
    \multicolumn{2}{c}{\footnotesize(b) \em Example 2}
     &\multicolumn{1}{c}{}\\
       \cline{1-2}  \cline{4-5}  
    $R$ & error && $R$ & error   \\ \cline{1-2}  \cline{4-5}   
    2 &  0.042157&& 2 & 0.042418   \\  \cline{1-2}  \cline{4-5}  
    4 &  0.019224&& 4 & 0.019347 \\  \cline{1-2}  \cline{4-5} 
    6 &  0.035701&&  6 & 0.035804    \\  \cline{1-2}  \cline{4-5}  
    8 &  0.064799&&  8 & 0.064946   \\  \cline{1-2}  \cline{4-5} 
    10 & 0.004741&&  10 & 0.004923  \\  \cline{1-2}  \cline{4-5}  
    12 & 0.003565&& 12 & 0.003803      \\  \cline{1-2}  \cline{4-5}  
    14 & 0.002979&&  14 & 0.003217    \\  \cline{1-2}  \cline{4-5}  
    16 & 0.002490&& 16 & 0.0028368     \\  \cline{1-2}  \cline{4-5}  
    18 & 0.001952&& 18 & 0.002430    \\  \cline{1-2}  \cline{4-5}   
    20 & 0.001035& & 20 & 0.001610   \\
 \cline{1-2}    \cline{4-5}
  \end{tabular}
\end{center}
\end{table}

\begin{figure}[h!]
\centering
\includegraphics[width=6.cm]{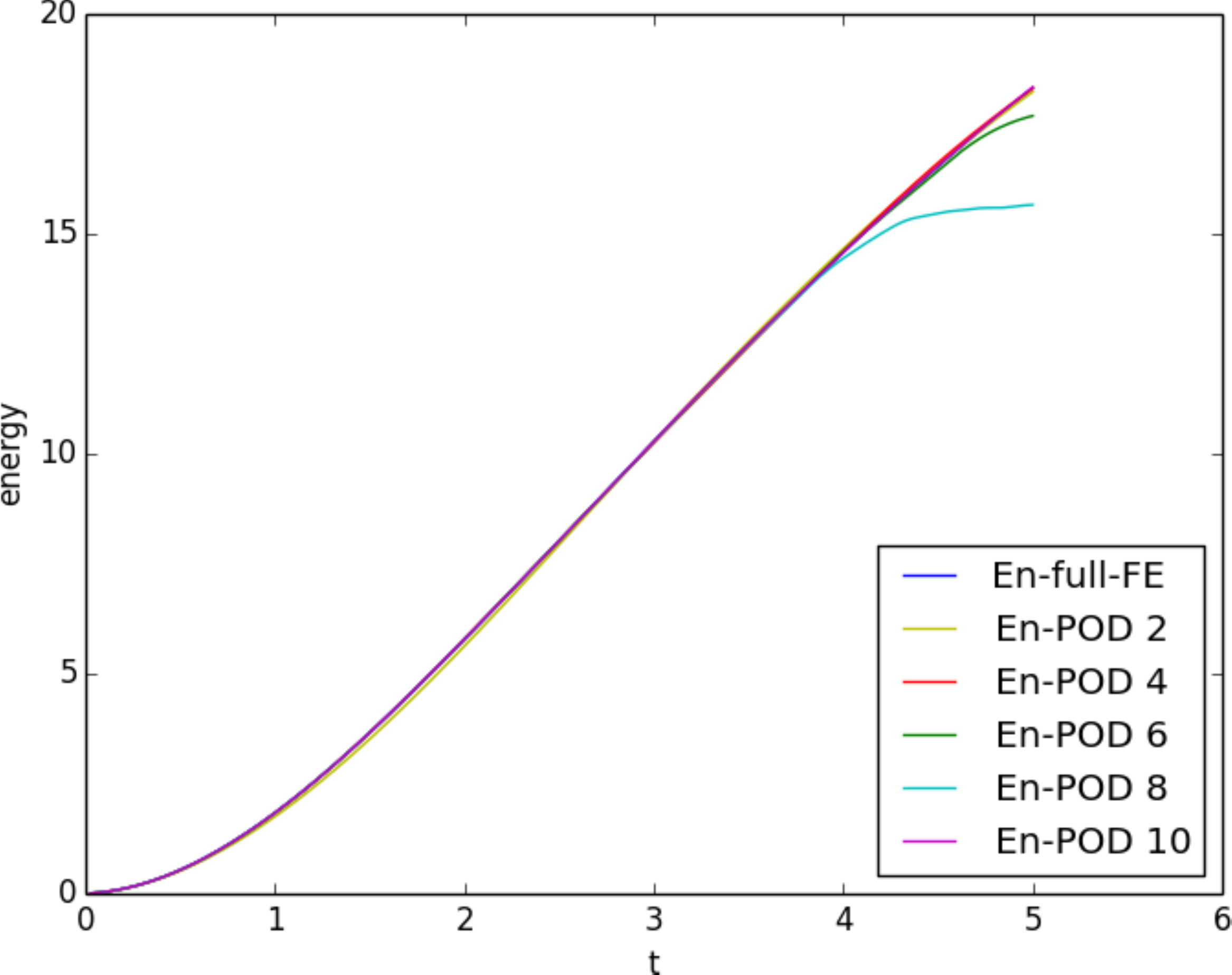}
\includegraphics[width=6.cm]{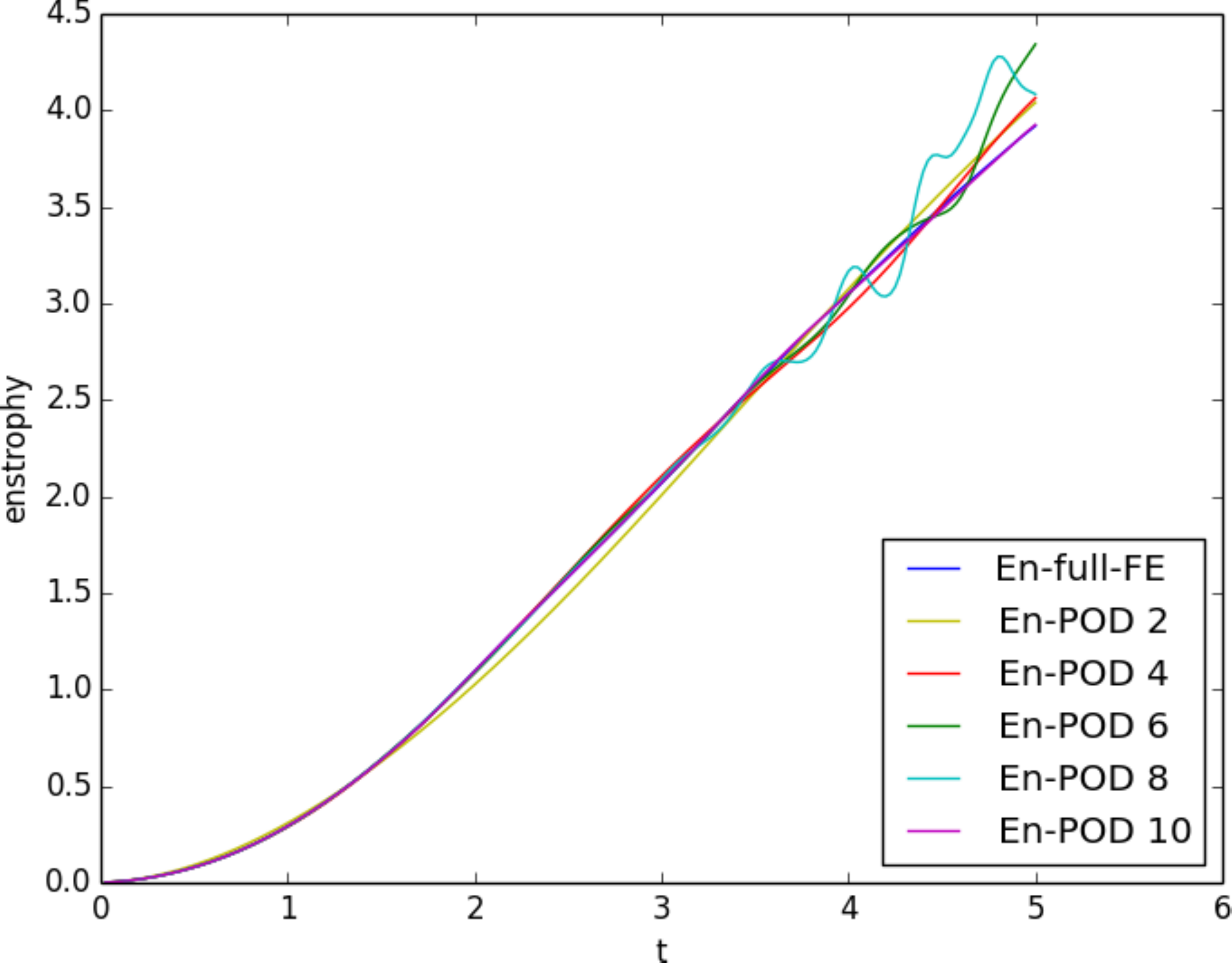}  
\caption{For Example 1 and for $0 \leq t \leq 5$, the energy (left) and enstrophy (right) of the ensemble determined for the En-full-FE approximations and for the En-POD approximation of several dimensions.}
\label{Energy_ex1}
\end{figure}

\subsection{Example 2}
Of course, the approximation of solutions of PDEs using reduced-order models such as POD are used not in the context of Section \ref{datam}, but, in our setting, for values of the perturbation parameter $\epsilon$ different from those used to generate the reduced-order basis. Thus, we consider the problem described in Section \ref{twoc} except that now we apply the {\em En-POD} method, using the basis generated as described in Section \ref{twoc}, for the two ensemble values $\epsilon_{1} = 0.1$ and $\epsilon_{2} = 1.0$, both of which are different from the values used to generate the snapshots used to construct the POD basis. For comparison purposes, we also determine the {\em En-full-FE} approximation for this ensemble. Note that these two values of $\epsilon$ take us to an {\em extrapolatory setting}, i.e., these values are outside of the interval $[-10^{-3},10^{-3}]$ bracketed by the values of $\epsilon$ used to generate the POD basis. Using a reduced-order method in an extrapolatory setting is usually a stern test of its efficacy.

The results for this ensemble are given in Table \ref{tabEx1}(b) and Figures \ref{noRomEx2}, \ref{errorDiff2}, and \ref{Energy_ex2}. The discussion in Section \ref{datam} corresponding to Example 1 carries over to this example except that the magnitude of the error is slightly larger; compare Table \ref{tabEx1}(a) and Table \ref{tabEx1}(b).

\begin{figure}[h!]
\centering
\includegraphics[height=5.5cm]{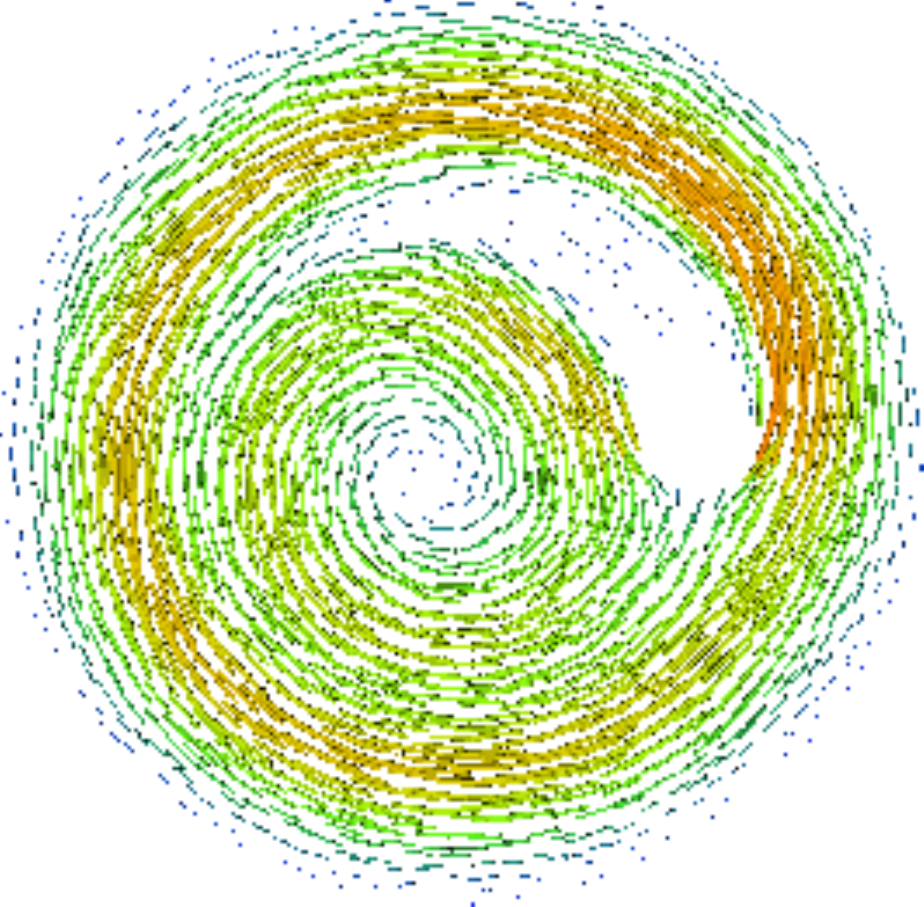}
\includegraphics[height=3cm]{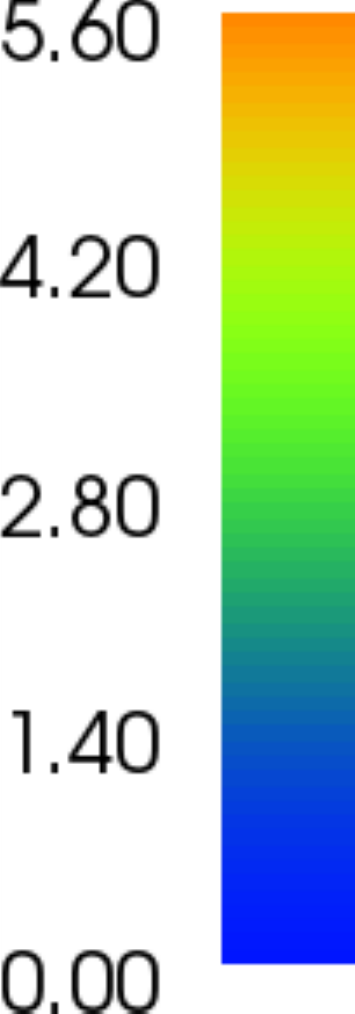} 
\includegraphics[height=5.5cm]{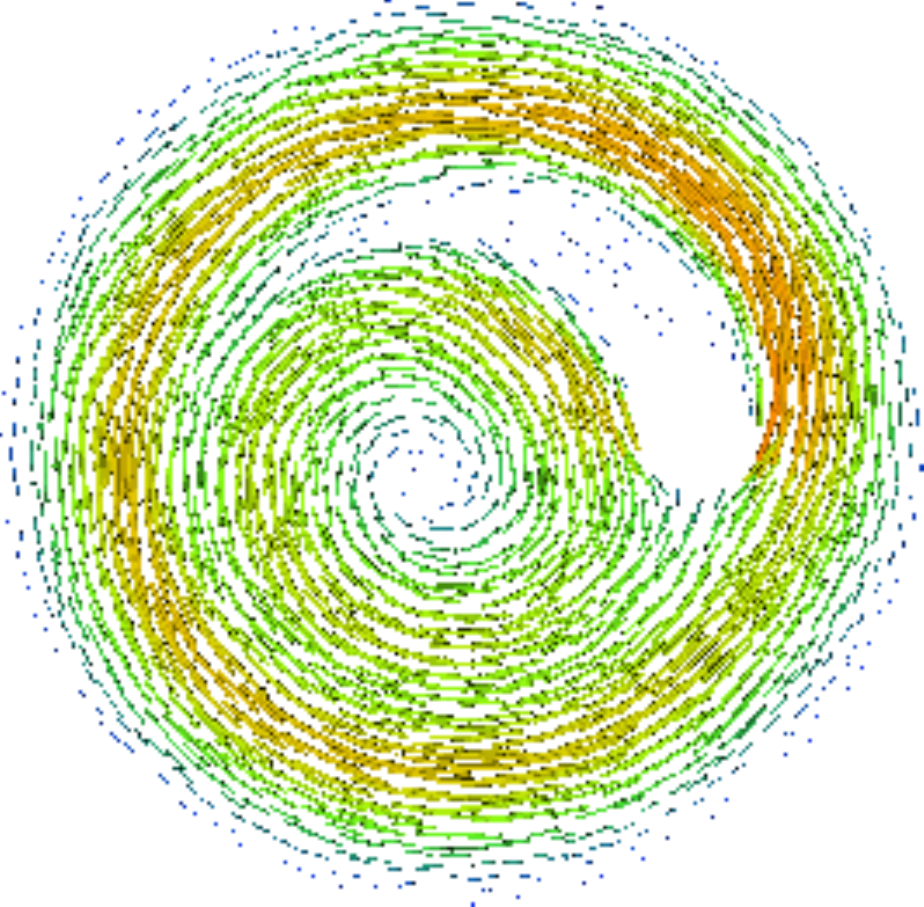} 
\caption{For Example 2, the ensemble average of the velocity field at the final time $T = 5$ of the En-full-FE (left) approximation and the En-POD approximation with $10$ reduced basis vectors (right).}
\label{noRomEx2}
\end{figure}

\begin{figure}[h!]
\centering
\includegraphics[height=5.5cm]{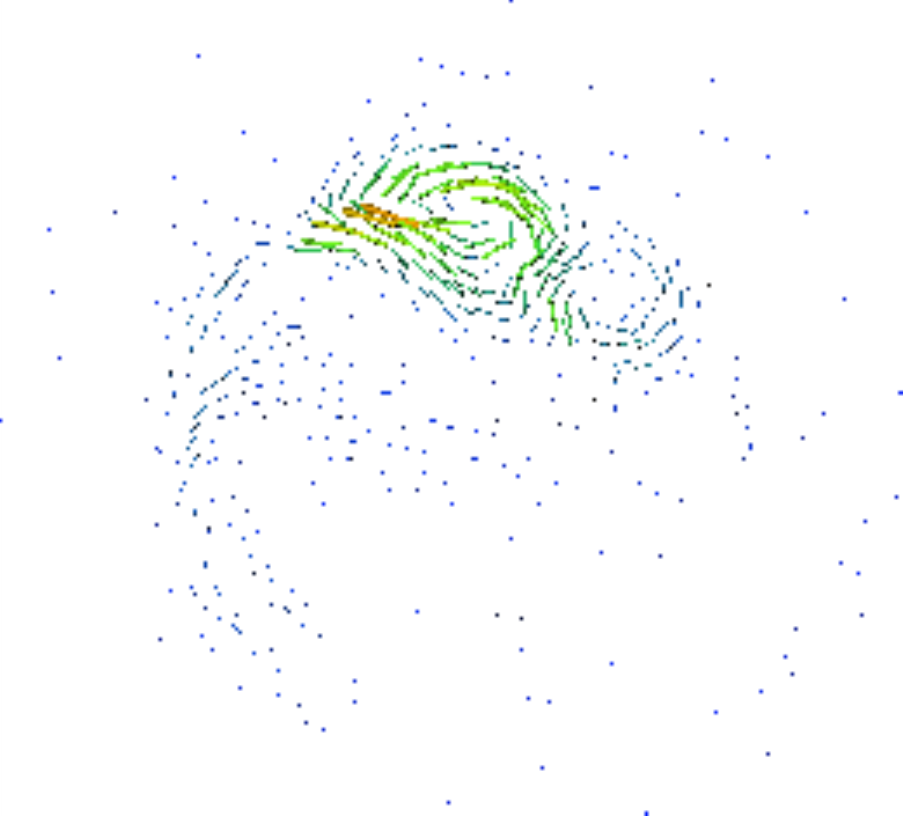}
\includegraphics[height=3cm]{scale2.pdf} 
\caption{For Example 2, the difference between the ensemble average of the velocity field at the final time $T = 5$ of the En-full-FE (approximation and the En-POD approximation with $10$ reduced basis vectors.}
\label{errorDiff2}
\end{figure}

\begin{figure}[h!]
\centering
\includegraphics[width=6.cm]{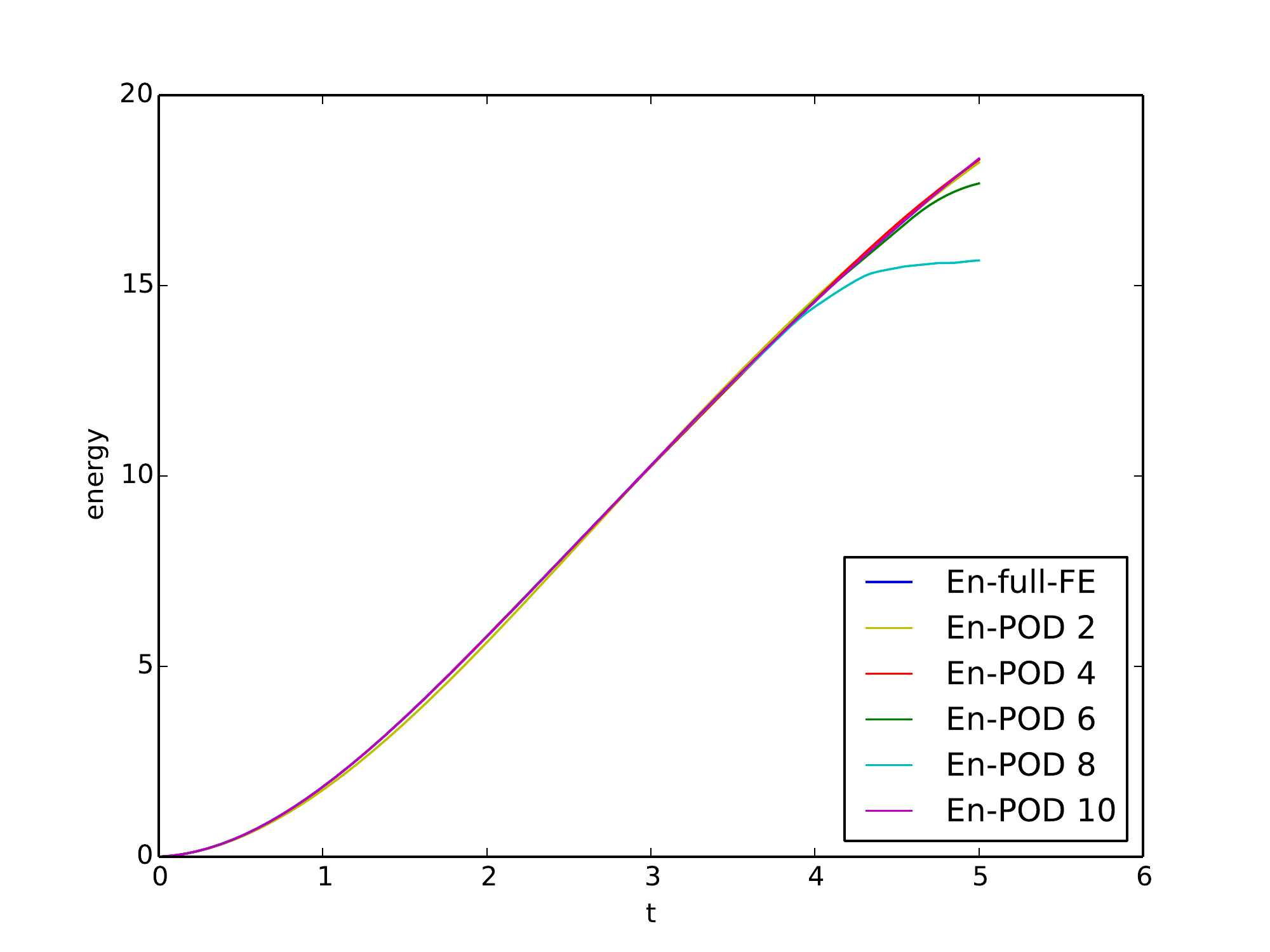}
\includegraphics[width=6.cm]{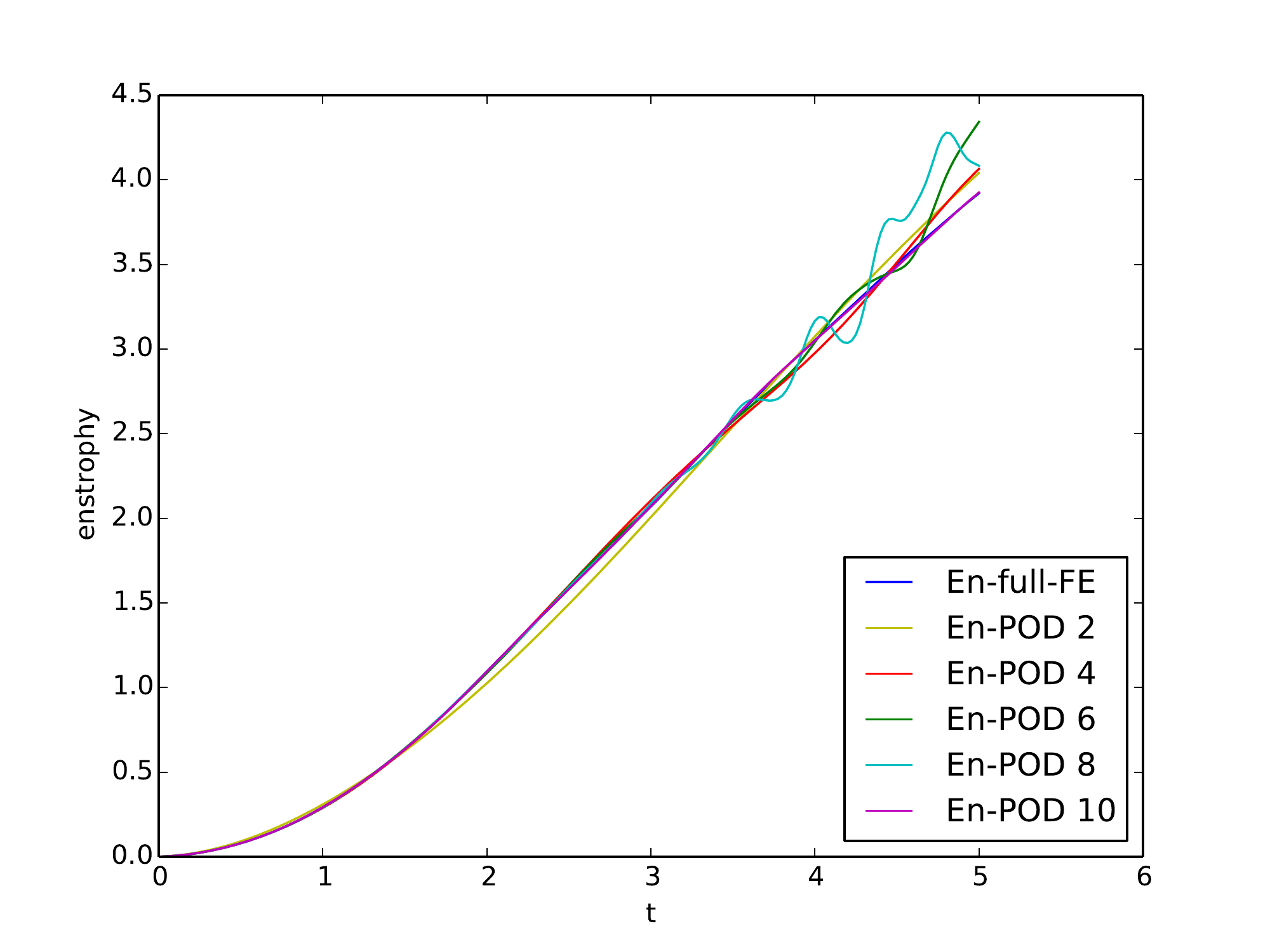}  
\vspace{-.1in}
\caption{For Example 2 and for $0 \leq t \leq 5$, the energy (left) and enstrophy (right) of the ensemble determined for the En-full-FE approximations and for the En-POD approximation of several dimensions.}
\label{Energy_ex2}
\end{figure}

\section{Concluding remarks}
In this work, an ensemble-proper orthogonal decomposition method for the nonstationary Navier-Stokes equations is proposed and analyzed. This method is built on a recently developed ensemble method that allows for the efficient determination of the multiple solutions of NSE. By incorporating the proper orthogonal decomposition technique, the ensemble-POD method introduced here significantly reduces the computational cost compared with that for the original ensemble method.

The method presented herein only works with low Reynolds number flows because the stability condition degrades quickly as the Reynolds number increases. To handle high Reynolds number flows, one has to consider incorporating regularization techniques. For single Navier-Stokes solves, there is existing in vast literature in this regard, but, in the ensemble setting, regularization has barely been studied. The only existing works are in \cite{J15,JL15}. The study of regularization methods in the ensemble and ensemble-POD setting is a focus of our current research.

We also note that in certain applications it may be desirable to construct a reduced basis for the pressure. We did not consider this in this work; doing so would require some sort of stabilization, such as the supremer stabilization introduced in \cite{BMQR15}, to compensate for the newly introduced LBB type condition. The incorporation of this type of method into the framework developed in this paper is also a subject of future research.

\appendix

\section{Proof or Theorem \ref{th:En-POD}}\label{app1}

{\allowdisplaybreaks

As the only difference is the choice of basis functions, we follow closely the proof of \cite[Theorem 1 (Stability of BEFE-Ensemble)]{JL14}.
Setting $\varphi=u_{R}^{j,n+1}$ in (\ref{En-POD-Weak}) and applying the Cauchy-Schwarz and Young inequalities to the right-hand side yields
\begin{equation}\label{ineq: tri}
\begin{aligned}
&\frac{1}{2}\|  u_{R}^{j,n+1}\| ^{2}-\frac{1}{2}\|  u_{R}^{j,n}\| 
^{2}+\frac{1}{2}\|  u_{R}^{j,n+1}-u_{R}^{j,n}\| ^{2}+ \nu\Delta t \| \nabla u_{R}^{j,n+1}\| ^{2}
\\&
+ \Delta t b^{*}(u_{R}^{j, n}-<u_R>^n,u_{R}^{j,n},u_{R}^{j,n+1}-u_{R}
^{j,n})\\&
\leq\frac{\nu\Delta t}{2} \| \nabla u_{R}^{j,n+1} \| ^{2}+ \frac{\Delta
t}{2 \nu} \|  f^{j,n+1}\| _{-1}^{2}.
\end{aligned}
\end{equation}
Next, we bound the trilinear term using the Poincar$\acute{e}$ inequality, Lemma \ref{lm:skew} and the inverse inequality \eqref{lm:inverse}:

\begin{equation}\label{Stab1}
\begin{aligned}
-  b^{*}&(u_{R}^{j, n}-<u_R>^n,u_{R}^{j,n},u_{R}^{j,n+1}-u_{R}
^{j,n})\\
&\leq C\| \nabla (u_{R}^{j, n}-<u_R>^n)\| \| \nabla u_{R}
^{j,n}\| \sqrt{\|  u_{R}^{j,n+1}-u_{R}^{j,n}\| \|  \nabla (u_{R}^{j,n+1}-u_{R}^{j,n})\|  }\\
&\qquad+\frac{1}{2}C\| \nabla\cdot (u_{R}^{j, n}-<u_R>^n)\| \| 
u_{R}^{j,n}\cdot(u_{R}^{j,n+1}-u_{R}^{j,n})\| \\
&\leq C\| \nabla (u_{R}^{j, n}-<u_R>^n)\| \| \nabla u_{R}
^{j,n}\| \sqrt{\|  u_{R}^{j,n+1}-u_{R}^{j,n}\| \|  \nabla (u_{R}^{j,n+1}-u_{R}^{j,n})\| }\\
&+\frac{1}{2}C\| \nabla\cdot (u_{R}^{j,n}-<u_R>^n)\| \| \nabla
u_{R}^{j,n}\| \sqrt{\|  u_{R}^{j,n+1}-u_{R}^{j,n}\|  \| 
\nabla(u_{R}^{j,n+1}-u_{R}^{j,n})\| }\\
&\leq C\| \nabla (u_{R}^{j, n}-<u_R>^n)\| \| \nabla u_{R}
^{j,n}\|  \|\hspace{-1pt}|  {\mathbb S}_R \|\hspace{-1pt}| _2^{1/4} \|\hspace{-1pt}|  {\mathbb M}_R^{-1} \|\hspace{-1pt}| _2^{1/4}\|  u_{R}^{j,n+1}-u_{R}^{j,n}\| \\
&+\frac{1}{2}C\| \nabla (u_{R}^{j, n}-<u_R>^n)\| \| \nabla
u_{R}^{j,n}\| \|\hspace{-1pt}|  {\mathbb S}_R \|\hspace{-1pt}| _2^{1/4} \|\hspace{-1pt}|  {\mathbb M}_R^{-1} \|\hspace{-1pt}| _2^{1/4}\|  u_{R}^{j,n+1}-u_{R}^{j,n}\| .
\end{aligned}
\end{equation}
By construction, the POD basis functions are orthonormal with respect to the $L^2(\Omega)$ inner product so that $\|\hspace{-1pt}|  {\mathbb M}_R \|\hspace{-1pt}| _2 = \|\hspace{-1pt}|  {\mathbb M}_R^{-1} \|\hspace{-1pt}| _2=1$. Then, \eqref{Stab1} reduces to 
$$
\begin{aligned}
-  b^{*}(u_{R}^{j, n}-&<u_R>^n,u_{R}^{j,n},u_{R}^{j,n+1}-u_{R}^{j,n})\\&
\leq C\| \nabla (u_{R}^{j,n}-<u_R>^n)\| \| \nabla u_{R}
^{j,n}\|  \|\hspace{-1pt}|  {\mathbb S}_R \|\hspace{-1pt}| _2^{1/4}\|  u_{R}^{j,n+1}-u_{R}^{j,n}\| .
\end{aligned}
$$
Using Young's inequality again results in
$$
\begin{aligned}
- \Delta t b^{*}(&u_{R}^{j, n}-<u_R>^n,u_{R}^{j,n},u_{R}^{j,n+1}-u_{R}
^{j,n})\\&
\leq C \Delta t^{2}\|\hspace{-1pt}|  {\mathbb S}_R \|\hspace{-1pt}| _2^{1/2} \| \nabla (u_{R}^{j, n}-<u_R>^n)\| ^{2}
\| \nabla u_{R}^{j,n}\| ^{2}+ \frac{1}{4}\|  u_{R}^{j,n+1}-u_{R}
^{j,n}\| ^{2}.
\end{aligned}
$$
Combining with (\ref{ineq: tri}) and then adding and subtracting $\frac{\nu\Delta t}{4}\| \nabla u_{R}^{j,n}\| ^{2}$ results in
$$
\begin{aligned}
\frac{1}{2}\|  u_{R}^{j,n+1}\| ^{2}-\frac{1}{2}\|  u_{R}^{j,n}\| 
^{2}&+\frac{1}{4}\|  u_{R}^{j,n+1}-u_{R}^{j,n}\| ^{2}+ \frac{\nu\Delta
t}{4}\Big\{\| \nabla u_{R}^{j,n+1}\| ^{2}-\| \nabla u_{R}^{j,n}\| 
^{2}\Big\}\\
+\frac{\nu\Delta t}{4}\Big\{\| \nabla u_{R}^{j,n+1}\| ^{2}&+\big(1-C \Delta
t {\nu}^{-1} \|\hspace{-1pt}|  {\mathbb S}_R \|\hspace{-1pt}| _2^{1/2}\|\nabla ( u_{R}^{j, n}-<u_R>^n)\| ^{2}\big) \| \nabla
u_{R}^{j,n}\| ^{2}\Big\} \\
&\leq\frac{\Delta t}{2 \nu} \|  f^{j,n+1}\| 
_{-1}^{2}.
\end{aligned}
$$
Assuming that the restriction (\ref{ineq:CFL-h}) holds, we have
$$
\frac{\nu\Delta t}{4}\big(1-C \Delta t {\nu}^{-1}\|\hspace{-1pt}|  {\mathbb S}_R \|\hspace{-1pt}| _2^{1/2}\| \nabla (u_{R}^{j,n}-<u_R>^n)\| ^{2}\big) \| \nabla u_{R}^{j,n}\| ^{2} \geq0.
$$
Combining the last two results then yields
$$
\begin{aligned}
\frac{1}{2}\|  u_{R}^{j,n+1}\| ^{2}&-\frac{1}{2}\|  u_{R}^{j,n}\| 
^{2}+\frac{1}{4}\|  u_{R}^{j,n+1}-u_{R}^{j,n}\| ^{2}
\\&
+ \frac{\nu\Delta t}{4}\Big\{\| \nabla u_{R}^{j,n+1}\| ^{2}-\| \nabla
u_{R}^{j,n}\| ^{2}\Big\} +\frac{\nu\Delta t}{4}\| \nabla u_{R}^{j,n+1}
\| ^{2} \leq\frac{\Delta t}{2 \nu} \|  f^{j,n+1}\| _{-1}^{2}.
\end{aligned}
$$
Summing up the above inequality results in \eqref{Stab:result}.
}

\section{Proof of Lemma \ref{lm:Projerr}}\label{app2}

{\allowdisplaybreaks

By \eqref{eq:def_proj} and the Cauchy-Schwarz inequality, we have 
$$
\begin{aligned}
\|  u^{m}-\Pi_R u^m\| ^2&= (u^{m}-\Pi_R u^m, u^{m}-\Pi_R u^m)\\
&=(u^{m}-\Pi_R u^m, u^{m}- \varphi) + (u^{m}-\Pi_R u^m,  \varphi-\Pi_R u^m)\\
&= (u^{m}-\Pi_R u^m, u^{m}- \varphi)\leq  \|  u^{m}-\Pi_R u^m \|  \|  u^{m}- \varphi\| 
\qquad\forall \varphi\in X_R
\end{aligned}
$$
so that
$$
\|  u^{m}-\Pi_R u^m\|  \leq  \|  u^{m}- \varphi\| \qquad\forall \varphi\in X_R. 
$$
We rewrite $u^{m}-\varphi=(u^m-u_S^{j,m})+(u_S^{j,m}-u^{j,m}_{h,S})+(u^{j,m}_{h,S}-\varphi)$ for all  $j=1,\ldots,J_S$. Setting $\varphi=\Pi_R u_{h,S}^{j,m}= \sum_{i=1}^R (u_{h,S}^{j,m}, \varphi_i)\varphi_i$ and using the triangle inequality as well as Lemma \ref{lm:L2err}, we have, for $j=1,\ldots,J_S$,
$$
\begin{aligned}
\frac{1}{N_S+1} &\sum_{m=0}^{N_S} \|  u^{m}-\Pi_R u^m\| ^2
\\&\leq \frac{1}{N_S+1} \sum_{m=0}^{N_S} (\|  u^{m}-u_S^{j,m}\| +\|  u_S^{j,m}-u_{h,S}^{j,m}\| +\|  u_{h,S}^{j,m}-\Pi_R u_{h,S}^{j,m}\| )^2
\\&
\leq \frac{2}{N_S+1} \sum_{m=0}^{N_S} \|  u^{m}-u_S^{j,m}\| ^2+C\Big( h^{2s+2} + \triangle t^4 \Big )+ 2J_S\sum_{i=R+1}^{J_S(N_S+1)} \lambda_i 
\end{aligned}
$$
from which \eqref{pel2} easily follows.
Similarly, by using Lemmas \ref{lm:inverse} and \ref{lm:H1err}, we have 
\begin{align*}
\frac{1}{N_S+1} &\sum_{m=0}^{N_S} \|  \nabla (u^{m}-\Pi_R u^m)\| ^2\\
&\leq \frac{1}{N_S+1} \sum_{m=0}^{N_S} \Big(\|  \nabla ( u^{m}-u_S^{j,m})\| +\|  \nabla (u_S^{j,m}-u_{h,S}^{j,m})\| 
\\&
\qquad +\|  \nabla (u_{h,S}^{j,m}-\Pi_R u_{h,S}^{j,m})\| +\|  \nabla (\Pi_R u_{h,S}^{j,m}-\Pi_R u^m)\|  \Big)^2\\
&\leq \frac{2}{N_S+1} \sum_{m=0}^{N_S} \|  \nabla ( u^{m}-u_S^{j,m})\| ^2+C \big ( h^{2s}+\triangle t^4 \big ) 
+2J_S\sum_{i=R+1}^{J_S(N_S+1)} \|  \nabla \varphi_i\| ^2\lambda_i \\
&\qquad+2\|\hspace{-1pt}|  {\mathbb S}_R \|\hspace{-1pt}| _2\frac{1}{N_S+1}\sum_{m=0}^{N_S} \Big (\|  \Pi_R u_{h,S}^{j,m}-\Pi_R u_S^{j,m}\| ^2+\|  \Pi_R u_S^{j,m}-\Pi_R u^m\| ^2 \Big )\\
&\leq \frac{2}{N_S+1} \sum_{m=0}^{N_S} \|  \nabla ( u^{m}-u_S^{j,m})\| ^2+C \big ( h^{2s}+\triangle t^4\big )
 +2 J_S\sum_{i=R+1}^{J_S(N_S+1)} \|  \nabla \varphi_i\| ^2\lambda_i \\
&\qquad +2\|\hspace{-1pt}|  {\mathbb S}_R \|\hspace{-1pt}| _2\frac{1}{N_S+1}\sum_{m=0}^{N_S} \Big (\|  u_{h,S}^{j,m}- u_S^{j,m}\| ^2+\|  u_S^{j,m}- u^m\| ^2 \Big )\\
&\leq \frac{2}{N_S+1} \sum_{m=0}^{N_S} \|  \nabla ( u^{m}-u_S^{j,m})\| ^2+C\Big( h^{2s} + \triangle t^4 \Big )
 +2J_S\sum_{i=R+1}^{J_S(N_S+1)} \|  \nabla \varphi_i\| ^2 \lambda_i \\
 &\qquad+2\|\hspace{-1pt}|  {\mathbb S}_R \|\hspace{-1pt}| _2 \Big (h^{2s+2} + \triangle t^4\Big )+\|\hspace{-1pt}|  {\mathbb S}_R \|\hspace{-1pt}| _2\frac{2}{N_S+1} \sum_{m=0}^{N_S} \|   u^{m}-u_S^{j,m} \| ^2
\end{align*}
from which \eqref{peh1} easily follows.
}

\section{Proof of Theorem \ref{th:errEn-POD}}\label{app3}

For $j=1,\cdots, J$, the true solutions of the NSE $u^{j}$ satisfies
{\allowdisplaybreaks
\begin{gather}
(\frac{u^{j,n+1}-u^{j,n}}{\Delta t}, \varphi)+ b^{*}(u^{j,n+1},
u^{j,n+1}, \varphi)+ \nu(\nabla u^{j, n+1}, \nabla v_{h})- (p^{j, n+1},\nabla\cdot \varphi)\label{eq:convtrue}\\
=(f^{j,n+1}, \varphi)+ Intp(u^{j,n+1};\varphi)\text{ , }\qquad\text{for any }
\varphi\in X_{R}\text{ ,}\nonumber
\end{gather}
where $Intp(u^{j,n+1};\varphi)$ is defined as
\[
Intp(u^{j,n+1};\varphi)=(\frac{u^{j,n+1}-u^{j,n}}{\Delta t}-u_{t}^j
(t^{n+1}),\varphi)\text{ .}
\]
Let
\begin{equation}
e^{j,n}=u^{j,n}-u_{R}^{j,n}=(u^{j,n}-\Pi_{R} u^{j,n})+(\Pi_{R} u^{j,n}-u_{R}^{j,n})=\eta^{j,n}+\xi_{R}^{j,n} \text{ ,\qquad} j=1,...,J\text{ ,}\nonumber
\end{equation}
where $\Pi_{R} u_{j}^{n} \in X_{R} $ is the $L^2$ projection of $u^{j,n}$
in $X_{R}.$ Subtracting (\ref{En-POD-Weak}) from (\ref{eq:convtrue}) gives
\begin{gather}
(\frac{\xi_{R}^{j,n+1}-\xi_{R}^{j,n}}{ \Delta t},\varphi) +\nu(\nabla\xi
_{R}^{j,n+1},\nabla \varphi)+b^{*}(u^{j,n+1},u^{j,n+1},\varphi)\nonumber\\
-b^{*}(<u_R>^n,u_{R}^{j,n+1},\varphi) -b^{*}(u_{R}^{j,n}-<u_R>^n,u_{R}^{j,n},\varphi)-(p^{j,n+1},\nabla\cdot \varphi)\label{eq:err}\\
=-(\frac{\eta^{j,n+1}-\eta^{j,n}}{\Delta t},\varphi) -\nu(\nabla\eta^{j,n+1},\nabla \varphi) +Intp(u^{j,n+1};\varphi)\text{ .}\nonumber
\end{gather}
Set $\varphi=\xi_{R}^{j,n+1}\in X_{R}$ and rearrange the nonlinear
terms. By the definition of the $L^2$ projection, we have $(\eta^{j,n+1}-\eta^{j,n}, \xi_{R}^{j,n+1})=0$. Thus \eqref{eq:err} becomes
\begin{gather}
\frac{1}{\Delta t}(\frac{1}{2}\|\xi_{R}^{j,n+1}\|^{2}-\frac{1}{2}\|\xi
_{R}^{j,n}\|^{2}+\frac{1}{2}\|\xi_{R}^{j,n+1}-\xi_{R}^{j,n}\|^{2})+\nu
\|\nabla\xi_{R}^{j,n+1}\|^{2}\nonumber\\
=-b^{*}(u^{j,n+1},u^{j,n+1},\xi_{R}^{j,n+1})+b^{*}(u_{R}^{j,n}
,u_{R}^{j,n+1},\xi_{R}^{j,n+1})\nonumber\\
-b^{*}(u_{R}^{j, n}-<u_R>^n,u_{R}^{j,n+1}-u_{R}^{j,n},\xi_{R}^{j,n+1})+(p^{j,n+1},\nabla\cdot\xi_{R}^{j,n+1})\label{eq:err1}\\
 -\nu
(\nabla\eta^{j,n+1},\nabla\xi_{R}^{j,n+1})+Intp(u^{j,n+1};\xi_{R}
^{j,n+1})\text{ .}\nonumber
\end{gather}
We rewrite the first two nonlinear terms on the right hand side of \eqref{eq:err1} as follows
\begin{equation}\label{eq:nonlinear}
\begin{aligned}
-b^{*}(u^{j,n+1},&u^{j,n+1},\xi_{R}^{j,n+1})+b^{*}(u_{R}^{j,n}
,u_{R}^{j,n+1},\xi_{R}^{j,n+1})\\
&= -b^{*}(e^{j,n}, u^{j,n+1}, \xi_{R}^{j,n+1})-
b^{*}(u_R^{j,n}, e^{j,n+1},\xi_{R}^{j,n+1})\nonumber\\
&\quad
-b^{*}(u^{j,n+1}-u^{j,n}, u^{j,n+1},\xi_{R}^{j,n+1})\nonumber\\
&=-b^{*}(\eta^{j,n}, u^{j,n+1},\xi_{R}^{j,n+1})
-b^{*}(\xi_{R}^{j,n}, u^{j,n+1},\xi_{R}^{j,n+1})\nonumber\\
&\quad-b^{*}(u_R^{j,n}, \eta^{j,n+1}, \xi_{R}^{j,n+1})
-b^{*}(u^{j,n+1}-u^{j,n}, u^{j,n+1},\xi_{R}^{j,n+1}).\nonumber
\end{aligned}
\end{equation}
Using the same techniques as in the proof of Theorem 5 of \cite{JL14}, with the assumption that $u^j \in L^{\infty}(0,T; H^1(\Omega))$, we have the following estimates on the nonlinear terms
\begin{equation}
\begin{aligned}
&b^{*}(u^{j,n+1}-u^{j,n},u^{j,n+1},\xi_{R}^{j,n+1})\\ 
&\leq\frac{\nu}{64}\|\nabla\xi_{R}^{j,n+1}\|^{2}+C\nu^{-1}\|\nabla(u
^{j,n+1}-u^{j,n})\|^{2}\|\nabla u^{j,n+1}\|^{2}\\
&\leq\frac{\nu}{64}\|\nabla\xi_{R}^{j,n+1}\|^{2}+\frac{C\Delta t}{\nu}
(\int_{t^{n}}^{t^{n+1}}\| \nabla u_{t}^j \|^{2} dt)\text{ , }
\end{aligned}
\end{equation}
and
\begin{equation}
\begin{aligned}
b^{*}(\eta^{j,n}, &u^{j,n+1}, \xi_{R}^{j,n+1}) 
\leq\frac{\nu}{64}\|\nabla\xi_{R}^{j,n+1}\|^{2}+C\nu^{-1}\|\nabla\eta^{j,n}
\|^{2} .
\end{aligned}
\end{equation}
Using Young's inequality, \eqref{In2} and the result \eqref{Stab:result} from the stability analysis, i.e., $\Vert \nabla u_R^{j,n}\Vert^2\leq C$, we have
\begin{equation}
\begin{aligned}
b^{*}(u^{j,n}_R,&\eta^{j,n+1},\xi_{R}^{j,n+1})\\
&\leq \Vert \nabla u_R^{j,n}\Vert^{1/2}\Vert u_R^{j,n}\Vert^{1/2}\Vert \nabla \eta^{j,n+1}\Vert \Vert \xi_R^{j,n+1}\Vert\\ 
&\leq\frac{\nu}{64}\|\nabla\xi_{R}^{j,n+1}\|^{2}+C\nu^{-1}\|\nabla u_R^{j,n}\|\|\nabla \eta^{j,n+1}\|^{2}.
\end{aligned}
\end{equation}
Using the inequality \eqref{In2}, Young's inequality and $u^j \in L^{\infty}(0,T; H^1(\Omega))$, we have
\begin{equation}
\begin{aligned}
b^{*}(\xi^{j,n}_{R}, &u^{j,n+1}, \xi_{R}^{j,n+1}) \\
&\leq C\| \nabla\xi
^{j,n}_{R} \|^{\frac{1}{2}} \| \xi^{j,n}_{R}\| ^{\frac{1}{2}}\|\nabla
u^{j,n+1}\|\|\nabla\xi_{R}^{j,n+1}\|\\
&\leq C(\epsilon\|\nabla\xi_{R}^{j,n+1}\|^{2}+\frac{1}{\epsilon}\|\nabla\xi
^{j,n}_{R} \|\|\xi^{j,n}_{R} \|)\\
&\leq C(\epsilon\|\nabla\xi_{R}^{j,n+1}\|^{2}+\frac{1}{\epsilon}(\delta
\|\nabla\xi^{j,n}_{R} \|^{2}+\frac{1}{\delta}\|\xi^{j,n}_{R} \|^2)\\
&\leq(\frac{\nu}{64} \|\nabla\xi_{R}^{j,n+1}\|^{2}+\frac{\nu}{8} \|\nabla
\xi^{j,n}_{R} \|^{2})+\frac{C}{\nu^{3}}\|\xi^{j,n}_{R} \|^{2} ,
\end{aligned}
\end{equation}
We next rewrite the third nonlinear term on the right-hand side of \eqref{eq:err1}:
\begin{equation}
\begin{aligned}
b^{*}(u_{R}^{j, n}-&<u_R>^n,u_{R}^{j,n+1}-u_{R}^{j,n},\xi_{R}^{j,n+1})\\
&=-b^{*}(u_{R}^{j, n}-<u_R>^n,e^{j,n+1}-e^{j,n},\xi_{R}^{j,n+1})\\
&\quad+b^{*}(u_{R}^{j,n}-<u_R>^n
,u^{j,n+1}-u^{j,n},\xi_{R}^{j,n+1})\\
&=-b^{*}(u_{R}^{j, n}-<u_R>^n,\eta^{j,n+1},\xi_{R}^{j,n+1})\\
&\quad+b^{*}(u_{R}^{j,n}-<u_R>^n,\eta
^{j,n},\xi_{R}^{j,n+1})\\
&\quad+b^{*}(u_{R}^{j,n}-<u_R>^n,\xi_{R}^{j,n},\xi_{R}^{j,n+1})\\
&\quad+b^{*}(u_{R}^{j,n}-<u_R>^n,u
^{j,n+1}-u^{j,n},\xi_{R}^{j,n+1})\text{ .}
\end{aligned}
\end{equation}
Following the same steps as in the proof of Theorem 5 of \cite{JL14}, we have the following estimates on the above nonlinear terms
\begin{equation}
\begin{aligned}
b^{*}(u_{R}^{j, n}-&<u_R>^n,\eta^{j,n+1},\xi_{R}^{j,n+1})\\ 
&\leq\frac{\nu}{64}\|\nabla\xi_{R}^{j,n+1}\|^{2}+C\nu^{-1}\|\nabla u_{R}
^{j, n}-<u_R>^n\|^{2}\|\nabla\eta^{j,n+1}\|^{2}\text{ ,}
\end{aligned}
\end{equation}
\begin{equation}
\begin{aligned}
b^{\ast}(u_{R}^{j,n}-&<u_R>^n,\eta^{j,n},\xi_{R}^{j,n+1})\\
&\leq\frac{\nu}{64}\| \nabla\xi_{R}^{j,n+1}\| ^{2}+C\nu^{-1}\| \nabla
u_{R}^{j, n}-<u_R>^n\| ^{2}\| \nabla\eta^{j,n}\| ^{2}\text{ .}
\end{aligned}
\end{equation}
By skew symmetry, Lemma \ref{lm:skew}, inequality \eqref{In1} and the inverse inequality \eqref{lm:inverse}, we have
\begin{equation}
\begin{aligned}
&b^{\ast}(u_{R}^{j, n}-<u_R>^n,\xi_{R}^{j,n},\xi_{R}^{j,n+1})\\
&\leq C\| \nabla u_{R}
^{j,n}-<u_R>^n\| \| \nabla\xi_{R}^{j,n+1}\| \sqrt{\| \xi_{R}^{j,n+1}-\xi_{R}
^{j,n}\| \|  \nabla (\xi_{R}^{j,n+1}-\xi_{R}
^{j,n})\| }\\
&+C\| \nabla\cdot (u_{R}^{j, n}-<u_R>^n)\| \| \xi_{R}^{j,n+1}\cdot
(\xi_{R}^{j,n+1}-\xi_{R}^{j,n})\| \\
&\leq C\| \nabla (u_{R}^{j,n}-<u_R>^n)\| \| \nabla\xi_{R}^{j,n+1}\| \sqrt{\| \xi_{R}^{j,n+1}-\xi_{R}
^{j,n}\| \|  \nabla (\xi_{R}^{j,n+1}-\xi_{R}
^{j,n})\| }\\
&\leq C\| \nabla (u_{R}^{j, n}-<u_R>^n)\| 
\| \nabla\xi_{R}^{j,n+1}\|  \|\hspace{-1pt}|  {\mathbb S}_R \|\hspace{-1pt}| _2^{1/4}\| \xi_{R}^{j,n+1}-\xi_{R}
^{j,n}\| \\
&\leq\frac{1}{4\triangle t}\| \xi_{R}^{j,n+1}-\xi_{R}^{j,n}\| ^{2}+\left(
C \triangle t \|\hspace{-1pt}|  {\mathbb S}_R \|\hspace{-1pt}| _2^{1/2}\| \nabla u_{R}^{j, n}-<u_R>^n\| ^{2}\right)  \| \nabla
\xi_{R}^{j,n+1}\| ^{2}.
\end{aligned}
\end{equation}
For the last nonlinear term we have
\begin{equation}
\begin{aligned}
b^{*}(u_{R}^{j, n}-&<u_R>^n,u^{j,n+1}-u^{j,n},\xi_{R}^{j,n+1}) \\
&\leq\frac{\nu}{64}\|\nabla\xi_{R}^{j,n+1}\|^{2}+C\nu^{-1}\|\nabla (u_{R}
^{j,n}-<u_R>^n)\|^{2}\|\nabla(u^{j,n+1}-u^{j,n})\|^{2}\\
&\leq\frac{\nu}{64}\|\nabla\xi_{R}^{j,n+1}\|^{2}+\frac{C \Delta t}{\nu}\|\nabla
(u_{R}^{j, n}-<u_R>^n)\|^{2}(\int_{t^{n}}^{t^{n+1}}\| \nabla u_{t}^j\|^{2} \text{
}dt)\text{ .}
\end{aligned}
\end{equation}
Next, consider the pressure term. Since $\xi_{R}^{j,n+1}\in X_{R}\subset V_h$
we have for $q_{h}\in Q_h$
\begin{equation}
\begin{aligned}
(p^{j,n+1},\nabla\cdot\xi_{R}^{j,n+1})&=(p^{j,n+1}-q_{h}^{n+1},
\nabla\cdot\xi_{R}^{j,n+1})\\
&\leq\frac{\nu}{64}\|\nabla\xi_{R}^{j,n+1}\|^{2}+C \nu^{-1}\|p
^{j,n+1}-q_{h}^{n+1}\|^{2} \text{ .}
\end{aligned}
\end{equation}
The other terms, are bounded as
\begin{equation}
\begin{aligned}
\nu(\nabla\eta^{j,n+1},\nabla\xi_{R}^{j,n+1}) 
\leq C\nu\|\nabla\eta^{j,n+1}\|^{2}+ \frac{\nu}{64}\|\nabla\xi_{R}
^{j,n+1}\|^{2} \text{ .}
\end{aligned}
\end{equation}
Finally,
\begin{equation}\label{lastineq}
\begin{aligned}
Intp(u^{j,n+1};\xi_{R}^{j,n+1})&=(\frac{u^{j,n+1}-u^{j,n}}{\Delta
t}-u_{t}^j(t^{n+1}),\xi_{R}^{j,n+1})\\
&\leq\frac{\nu}{64}\|\nabla\xi_{R}^{j,n+1}\|^{2}+\frac{C}{\nu}\|\frac
{u^{j,n+1}-u^{j,n}}{\Delta t}-u_{t}^j(t^{n+1})\|^{2}\\
&\leq\frac{\nu}{64}\|\nabla\xi_{R}^{j,n+1}\|^{2}+\frac{C\Delta t}{\nu}
\int_{t^{n}}^{t^{n+1}}\|u_{tt}^j\|^{2} dt \text{ .}
\end{aligned}
\end{equation}
Combining, we now have the following inequality:
\begin{equation}\label{eq:err2}
\begin{aligned}
\frac{1}{\Delta t}&\left(\frac{1}{2}\|\xi_{R}^{j,n+1}\|^{2}-\frac{1}{2}\|\xi
_{R}^{j,n}\|^{2}+\frac{1}{4}\| \xi_{R}^{j,n+1}-\xi_{R}^{j,n}\| ^{2}
\right)+\frac{\nu}{8}\left(\| \nabla\xi_{R}^{j,n+1}\| ^{2}-\| \nabla\xi_{R}
^{j,n}\| ^{2}\right)\\
&+\left(\frac{\nu}{4}-C\triangle t \|\hspace{-1pt}|  {\mathbb S}_R \|\hspace{-1pt}| _2^{1/2}\| \nabla (u_{R}^{j, n}-<u_R>^n)\| ^{2}\right)\| \nabla
\xi_{R}^{j,n+1}\| ^{2}\\
&\leq \frac{C}{\nu^{3}}\| \xi_{R}^{j,n}\| ^{2}+C\nu^{-1}\Vert \nabla u_R^{j,n}\Vert\| \nabla\eta^{j,n+1}
\| ^{2}+C\nu^{-1}\| \nabla\eta^{j,n}
\| ^{2}\\
&\quad
+\frac{C\Delta t}{\nu}\left(\int_{t^{n}}^{t^{n+1}}\| \nabla u_{t}^j
\| ^{2}dt\right)+C\nu\| \nabla\eta^{j,n+1}\| ^{2}\\
&\quad+C\nu^{-1}\| \nabla
\left(u_{R}^{j, n}-<u_R>^n\right)\| ^{2}\| \nabla\eta^{j,n+1}\| ^{2}+C\nu^{-1}\|  p^{j,n+1}-q_{h}^{n+1}
\| ^{2}\\
&\quad+C\nu^{-1}\| \nabla \left(u_{R}^{j,n}-<u_R>^n\right)\| ^{2}\| \nabla\eta^{j,n}\| 
^{2} +\frac{C\Delta t}{\nu}\int_{t^{n}
}^{t^{n+1}}\|  u_{tt}^j\| ^{2}dt\\
&\quad+\frac{C\Delta t}{\nu}\| \nabla \left(u_{R}^{j,n}-<u_R>^n\right)\| ^{2}\left(\int_{t^{n}}^{t^{n+1}
}\| \nabla u_{t}^j\| ^{2}dt\right)\text{ .}
\end{aligned}
\end{equation}
By the timestep condition \ $\frac{\nu}{4}-C\triangle t \|\hspace{-1pt}|  {\mathbb S}_R \|\hspace{-1pt}| _2^{1/2}\| \nabla
(u_{R}^{j, n}-<u_R>^n)\| ^{2}>0$. Take the sum of (\ref{eq:err2}) from $n=0$ to
$n=N-1$ and multiply through by $\Delta t$. Since $u_R^{j,0}=\sum_{i=1}^R \left(u^{j,0}, {\varphi}_i\right)\varphi_i$, we have $\|  \xi_{R}^{j,0}\| ^2=0$ and  $\|  \nabla \xi_{R}^{j,0}\| ^2=0$. 
\begin{equation}\label{ineq:err3}
\begin{aligned}
\frac{1}{2}\|&\xi_{R}^{j,N}\|^{2}+\frac{\nu\Delta t}{8}\|\nabla \xi_{R}
^{j,N}\|^{2}+\sum_{n=0}^{N-1}\frac{1}{4}\| \xi_{R}^{j,n+1}-\xi_{j,R}^{n}
\| ^{2}+C\Delta t\sum_{n=0}^{N-1}\nu\| \nabla\xi_{R}^{j,n+1}\| 
^{2}\\
&\leq \Delta t\sum_{n=0}^{N-1}\frac{C}{\nu^{3}}\| \xi_{R}
^{j,n}\| ^{2}
+\Delta t\sum_{n=0}^{N-1}\Bigg\{C\nu^{-1}\Vert \nabla u_R^{j,n}\Vert\| \nabla\eta^{j,n+1}
\| ^{2} +C\nu^{-1}\| \nabla\eta^{j,n}
\| ^{2}\\
&\quad
+\frac{C\Delta t}{\nu}\left(\int_{t^{n}}^{t^{n+1}}\| \nabla u_{t}^j
\| ^{2}dt\right)+C\nu\| \nabla\eta^{j,n+1}\| ^{2}+\frac{C\Delta t}{\nu}\int_{t^{n}
}^{t^{n+1}}\|  u_{tt}^j\| ^{2}dt\\
&\quad+C\Delta t^{-1}\|\hspace{-1pt}|  {\mathbb S}_R \|\hspace{-1pt}| _2^{-1/2}\| \nabla\eta^{j,n+1}\| ^{2}+C\nu^{-1}\Delta t^{-1}|\hspace{-1pt}|  {\mathbb S}_R \|\hspace{-1pt}| _2^{-1/2}\| \nabla\eta^{j,n}\| 
^{2} \\
&\quad+C\nu^{-1}\|  p^{j,n+1}-q_{h}^{n+1}
\| ^{2}+C\|\hspace{-1pt}|  {\mathbb S}_R \|\hspace{-1pt}| _2^{-1/2}(\int_{t^{n}}^{t^{n+1}
}\| \nabla u_{t}^j\| ^{2}dt)\Bigg\}\text{ .}
\end{aligned}
\end{equation}
Using the result \eqref{Stab:result} from the stability analysis, i.e., $\Delta t \sum_{n=0}^{N-1}\nu\Vert \nabla u_R^{j,n}\Vert^2\leq C$ and Assumption \eqref{assumption1}, we have
\begin{align}
C\nu^{-1}\Delta t&\sum_{n=0}^{N-1}\Vert \nabla u_R^{j,n}\Vert\| \nabla\eta^{j,n+1}
\| ^{2}\\
&\leq C\nu^{-2} \Bigg(\inf_{{j\in\{1,\ldots,J_S\}}}\frac{1}{N_S} \sum_{m=1}^{N_S} (\|  \nabla (u^{m}-u_S^{j,m})\|+\|\hspace{-1pt}|  {\mathbb S}_R \|\hspace{-1pt}| _2 \|   u^{m}-u_S^{j,m}\| ^2  ) ^2\nonumber\\
&
+(C+h^2 \|\hspace{-1pt}|  {\mathbb S}_R \|\hspace{-1pt}| _2 ) h^{2s} + (C+\|\hspace{-1pt}|  {\mathbb S}_R \|\hspace{-1pt}| _2 )\triangle t^4  + J_S\sum_{i=R+1}^{J_SN_S} \|  \nabla \varphi_i\| ^2\lambda_i \Bigg) \nonumber
\end{align}
Now applying Lemma \ref{lm:Projerr} gives
\begin{equation}\label{ineq:errlast}
\begin{aligned}
\frac{1}{2}\|\xi_{R}^{j,N}&\|^{2}+\sum_{n=0}^{N-1}\frac{1}{4}\| \xi
_{R}^{j,n+1}-\xi_{R}^{j,n}\| ^{2}+\frac{\nu\Delta t}{8}\|\nabla\xi
_{R}^{j,N}\|^{2}+C\Delta t\sum_{n=0}^{N-1}\nu\| \nabla\xi_{R}^{j,n+1}
\| ^{2}\\
&\leq  \Delta t\sum_{n=0}^{N-1}\frac{C}{\nu^{3}}\| \xi_{R}
^{j,n}\| ^{2}+\Bigg(C\nu^{-2}+C N_S \triangle t\nu^{-1}+CN_S\|\hspace{-1pt}|  {\mathbb S}_R \|\hspace{-1pt}| _2^{-1/2}\\
&+CN_S\nu \triangle t+C\nu^{-1}N_S\|\hspace{-1pt}|  {\mathbb S}_R \|\hspace{-1pt}| _2^{-1/2}
 \Bigg) \cdot \Bigg( \inf_{{j\in\{1,\ldots,J_S\}}}\frac{1}{N_S} \sum_{m=1}^{N_S} (\|  \nabla (u^{m}-u_S^{j,m})\| ^2\\
&+\|\hspace{-1pt}|  {\mathbb S}_R \|\hspace{-1pt}| _2 \|   u^{m}-u_S^{j,m}\| ^2  )
+(C+h^2 \|\hspace{-1pt}|  {\mathbb S}_R \|\hspace{-1pt}| _2 ) h^{2s} + (C+\|\hspace{-1pt}|  {\mathbb S}_R \|\hspace{-1pt}| _2 )\triangle t^4  \\
&+ J_S\sum_{i=R+1}^{J_SN_S} \|  \nabla \varphi_i\| ^2\lambda_i  \Bigg )
+C  \Delta t \|\hspace{-1pt}|  {\mathbb S}_R \|\hspace{-1pt}| _2^{-1/2}\| |\nabla
u_{t}^j|\| _{2,0}^{2}\\
&+C\frac{h^{2s}}{\nu}\| |p^j|\| _{2,s}^{2}
+\frac{C{\Delta t}^{2}}{\nu
}\| |u_{tt}^j|\| _{2,0}^{2}+\frac{C\Delta t^{2}}{\nu}\| |\nabla u_{t}^j|\| _{2,0}^{2}.
\end{aligned}
\end{equation}
The next step will be the application of the discrete Gronwall
inequality (Girault and Raviart \cite{GR79}, p. 176).
\begin{equation}\label{ineq:errlast1}
\begin{aligned}
\frac{1}{2}\| &\xi_{R}^{j,N}\| ^{2}+\sum_{n=0}^{N-1}\frac{1}{4}\| 
\xi_{R}^{j,n+1}-\xi_{R}^{j,n}\| ^{2}+\frac{\nu\Delta t}{8}\| \nabla
\xi_{R}^{j,N}\| ^{2}+C\Delta t\sum_{n=0}^{N-1}\nu\| \nabla\xi_{R}
^{j,n+1}\| ^{2}\\
&\leq \exp \Big(\frac{CT}{\nu^{3}}\Big)\Bigg\{\Bigg(C\nu^{-2}+C N_S \triangle t\nu^{-1}+CN_S\|\hspace{-1pt}|  {\mathbb S}_R \|\hspace{-1pt}| _2^{-1/2}\\
&+CN_S\nu \triangle t+C\nu^{-1}N_S\|\hspace{-1pt}|  {\mathbb S}_R \|\hspace{-1pt}| _2^{-1/2}
 \Bigg) \cdot \Bigg( \inf_{{j\in\{1,\ldots,J_S\}}}\frac{1}{N_S} \sum_{m=1}^{N_S} (\|  \nabla (u^{m}-u_S^{j,m})\| ^2\\
&+\|\hspace{-1pt}|  {\mathbb S}_R \|\hspace{-1pt}| _2 \|   u^{m}-u_S^{j,m}\| ^2  )
+(C+h^2 \|\hspace{-1pt}|  {\mathbb S}_R \|\hspace{-1pt}| _2 ) h^{2s} + (C+\|\hspace{-1pt}|  {\mathbb S}_R \|\hspace{-1pt}| _2 )\triangle t^4  \\
&+ J_S\sum_{i=R+1}^{J_SN_S} \|  \nabla \varphi_i\| ^2\lambda_i  \Bigg )
+C  \Delta t \|\hspace{-1pt}|  {\mathbb S}_R \|\hspace{-1pt}| _2^{-1/2}\| |\nabla
u_{t}^j|\| _{2,0}^{2}\\
&+C\frac{h^{2s}}{\nu}\| |p^j|\| _{2,s}^{2}
+\frac{C{\Delta t}^{2}}{\nu
}\| |u_{tt}^j|\| _{2,0}^{2}+\frac{C\Delta t^{2}}{\nu}\| |\nabla u_{t}^j|\| _{2,0}^{2}\Bigg\}\text{ .}
\end{aligned}
\end{equation}
Recall that $e^{j,n}=\eta^{j,n}+\xi_{R}^{j,n}$. To simplify formulas, we drop the second and third term on the left hand side of \eqref{ineq:errlast1}. Then by the triangle
inequality and Lemma \ref{lm:Projerr}, absorbing constants, we have \eqref{ineq:err00}.

}

\end{document}